\documentclass[11pt]{article}
\usepackage{amssymb,amsmath,amsfonts,amsthm,mathrsfs}
\usepackage{graphicx,graphics,psfrag,epsfig,calrsfs,cancel}
\usepackage[colorlinks=true, pdfstartview=FitV, linkcolor=blue, citecolor=red, urlcolor=blue]{hyperref}

\usepackage{float}
\usepackage{caption,subfig,color}
\usepackage{tikz}
\usetikzlibrary{calc}

\usepackage{geometry}
\newtheorem{theorem}{Theorem}[section]

\newtheorem{proposition}[theorem]{Proposition}

\newtheorem{assumption}[theorem]{Assumption}
\newtheorem{lemma}[theorem]{Lemma}

\numberwithin{equation}{section}
\numberwithin{figure}{section}

\newcommand{\N}{\mathbb{N}}
\newcommand{\Z}{\mathbb{Z}}
\newcommand{\R}{\mathbb{R}}
\newcommand{\C}{\mathbb{C}}
\newcommand{\eps}{\varepsilon}

\begin{document}

\title{The Neumann numerical boundary condition \\
for transport equations}

\author{Jean-Fran\c{c}ois {\sc Coulombel}\thanks{Institut de Math\'ematiques de Toulouse ; 
UMR5219, Universit\'e de Toulouse ; CNRS, Universit\'e Paul Sabatier, F-31062 Toulouse Cedex 9, 
France. Email: {\tt jean-francois.coulombel@math.univ-toulouse.fr}. Research of the author was supported 
by ANR project BoND, ANR-13-BS01-0009, and by the ANR project NABUCO, ANR-17-CE40-0025.} 
$\,$ \& Fr\'ed\'eric {\sc Lagouti\`ere}\thanks{Universit\'e de Lyon, Universit\'e Claude Bernard Lyon 1, 
Institut Camille Jordan (CNRS UMR5208), 43 boulevard du 11 novembre 1918, F-69622 Villeurbanne Cedex, 
France. Email: {\tt lagoutiere@math.univ-lyon1.fr}. Research of the author was supported by ANR project BoND, 
ANR-13-BS01-0009, and by the ANR project NABUCO, ANR-17-CE40-0025.}}
\date{\today}
\maketitle

\begin{abstract}
In this article, we show that prescribing homogeneous Neumann type numerical boundary conditions at an outflow boundary 
yields a convergent discretization in $\ell^\infty$ for transport equations. We show in particular that the Neumann numerical 
boundary condition is a stable, local, and absorbing numerical boundary condition for discretized transport equations. Our 
main result is proved for explicit two time level numerical approximations of transport operators with arbitrarily wide stencils. 
The proof is based on the energy method and bypasses any normal mode analysis.
\end{abstract}




\section{Introduction}

It is a well-known fact that transport equations do not require prescription of any boundary condition at an \emph{outflow} boundary, 
that is, when the transport velocity is outgoing with respect to the boundary of the spatial domain. This can be easily understood by 
integrating the equation along the characteristics. However, many discretizations of the transport equation involve a stencil that 
includes cells of the numerical grid that are located in the \emph{downstream} region. Such discretizations necessitate the prescription 
of numerical boundary conditions at an outflow boundary \cite{gko,strikwerda}, even though the underlying partial differential operator 
does not require any boundary condition for determining the solution.

The construction and analysis of transparent {\em versus} absorbing numerical boundary conditions for wave propagation problems now 
has a very long history, going back at least to the fundamental contribution by Engquist and Majda \cite{engquistmajda}, see also among 
numerous other works \cite{hagstrom,halpern,higdon1,higdon3,gko,aabes} and references therein. Our goal in this article is to rigorously 
justify the commonly acknowledged fact that enforcing Neumann type boundary conditions at an outflow boundary ``\emph{does the job}'', 
in the simple one-dimensional case with a constant velocity (both on the half-line and on an interval). In several respects, the result is 
definitely not new. Stability estimates for Neumann numerical boundary conditions were stated -though without proof- for instance by Kreiss 
\cite{kreissproc}, and a detailed proof of the latter result was later provided by Goldberg \cite{goldberg}. This result now enters the more 
general framework of \cite{goldberg-tadmor}. However, the approach in \cite{kreissproc,goldberg,goldberg-tadmor} is rather elaborate since 
it relies first on the verification of the so-called Uniform Kreiss-Lopatinskii Condition (that is, in the present context of numerical schemes, a 
refined version of the Godunov-Ryabenkii condition \cite{gks,gko}), and then on the application of deep general results which show that the 
latter condition is sufficient for the derivation of optimal semigroup estimates. Such general results first arose in \cite{kreissproc} and were 
later proved in further generality by Kreiss, Osher and followers \cite{kreiss1,osher1,wu,jfcag}. When combined with the trace estimates 
provided by the fulfillment of the Uniform Kreiss-Lopatinskii Condition and the general convergence result of Gustafsson \cite{gustafsson}, 
one gets a complete -though lengthy (!) and somehow unclear as far as the topology is concerned- justification that enforcing Neumann 
type numerical boundary conditions at an outflow boundary yields a convergent scheme for discretized transport equations.

Our goal here is to bypass all the arguments of those previous works that were based on the normal mode analysis and to obtain a more 
direct convergence result \emph{in the $\ell^\infty$ topology} with arguments that are as elementary as possible. Our approach is based 
on the energy method and discrete integration by parts rather than on the Laplace transform. We hope that some of our arguments might 
be useful to deal with more involved problems such as multidimensional hyperbolic systems. The main gain when enforcing the Neumann 
numerical boundary condition with respect, for instance, to the Dirichlet numerical boundary condition is to obtain convergence results in 
$\ell^\infty$, while in the case of the Dirichlet boundary condition at an outflow boundary, despite unconditional stability properties 
\cite{goldberg-tadmor}, boundary layer phenomena allow at best for convergence results in $\ell^p$, $p<+\infty$, only 
\cite{kreisslundqvist,chainais-grenier,bbjfc}. Our main result below gives a rate of convergence in $\ell^\infty$ for such discretizations of the 
transport equations. We do not claim that our rate is optimal, but we do not assume the discretization of the transport equation to be stable 
in $\ell^\infty(\Z)$ either, so if we do not reach optimality we are certainly not too far from it. We plan to study the more favorable case of 
$\ell^\infty$ stable schemes in the future, with the aim of improving the rate of convergence.

The rest of this article is organized as follows. In Section \ref{sect:2}, we introduce some notation and state our main convergence result for 
Neumann numerical boundary conditions at an outflow boundary. Based on the standard approach of numerical analysis, our convergence 
result relies on accurate stability estimates and a consistency analysis. Our main stability estimate, based on a new elementary approach, 
is stated and proved in Section \ref{sect:3}. (A crucial discrete integration by parts lemma is given in Appendix \ref{appA}.) The concluding 
arguments for proving our main result are given in Section \ref{sect:4}.

\section{Main result}
\label{sect:2}

In this article, we consider the following transport problem on an interval. We are given a fixed constant velocity $a>0$, an interval length 
$L>0$ and consider the problem:
\begin{equation}
\label{transport}
\begin{cases}
\partial_t u +a \, \partial_x u =0 \, ,& t \ge 0 \, ,\, x \in (0,L) \, ,\\
u(0,x)=u_0(x) \, ,& x \in (0,L) \, ,\\
u(t,0) =0 \, ,& t \ge 0 \, ,
\end{cases}
\end{equation}
with, at least, $u_0 \in L^2((0,L))$. Actually, further regularity and compatibility requirements will be enforced later on, but let us stick to that 
simple framework for the moment. We could consider an inhomogeneous Dirichlet boundary condition at $x=0$ in \eqref{transport}, but we 
rather stick to this simpler case in order to highlight the main novelty of our approach which rather focuses on the downstream boundary 
$x=L$ of the interval.

The solution to \eqref{transport} is given by the method of characteristics, which yields the explicit representation formula:
\begin{equation}
\label{soltransport}
\forall \, (t,x) \in \R^+ \times (0,L) \, ,\quad u(t,x)=u_0(x-a\, t) \, ,
\end{equation}
where it is understood in \eqref{soltransport} that the initial condition $u_0$ has been extended by zero 
to $\R^-$ (no extension is needed on $(L,+\infty)$ since $a$ is positive and therefore $x-a\, t <L$ for all 
relevant values of $t$ and $x$).

It may seem a too much trivial problem to approximate the problem \eqref{transport} for which an explicit solution is given, but one should 
keep in mind that such a representation formula ceases to be available for hyperbolic \emph{systems} in \emph{several} space dimensions, 
and our goal is to develop analytical tools which do not rely on the fact that \eqref{transport} is a \emph{one-dimensional scalar} problem. We 
therefore consider from now on an approximation of \eqref{transport} by means of a finite difference scheme. We are given a positive integer 
$J$, that is meant to be large, and define accordingly the space step $\Delta x$ and the grid points $(x_j)$ by:
$$
\Delta x :=L/J \, ,\quad x_j := j \, \Delta x \quad (j \in \Z) \, .
$$
The time step $\Delta t$ is then defined as $\Delta t:=\lambda \, \Delta x$ where $\lambda>0$ is a \emph{fixed} constant that is tuned so that 
our main Assumption \ref{as:scheme} below is satisfied. The interval $(0,L)$ is divided in $J$ cells $(x_{j-1},x_j)$, $j=1,\dots,J$, as depicted 
in Figure \ref{fig:maillage}. In what follows, we use the notation $t^n := n \, \Delta t$, $n \in \N$, and $u_j^n$ will play the role of an approximation 
for the solution $u$ to \eqref{transport} at time $t^n$ on the cell $(x_{j-1},x_j)$ (or at the mid-point $(x_{j-1}+x_j)/2$). We do not wish to discriminate 
between finite difference or finite volume schemes for \eqref{transport}, so rather than deriving this or that type of numerical scheme, we consider a 
linear iteration for the $u_j^n$'s that reads in the \emph{interior} domain:
\begin{equation}
\label{transportnum}
u_j^{n+1} =\sum_{\ell=-r}^p a_\ell \, u_{j+\ell}^n \, ,\quad n \in \N \, ,\quad j=1,\dots,J \, .
\end{equation}
In \eqref{transportnum}, $r,p$ are fixed nonnegative integers, and the coefficients $a_\ell$, $\ell=-r,\dots,p$ may only depend on the ratio $\lambda$ 
and the velocity $a$. Most of the usual linear explicit schemes, such as the upwind, Rusanov, Lax-Friedrichs and Lax-Wendroff schemes, 
can be put in that form. The interior domain corresponds to the indices $j=1,\dots,J$. However, in \eqref{transportnum}, the determination 
of $(u_1^{n+1},\dots,u_J^{n+1})$ requires the prior knowledge of $(u_{1-r}^n,\dots,u_{J+p}^n)$, which corresponds to a larger set of cells. 
In what follows the cells $(x_{j-1},x_j)$ with $j=1-r,\dots,0$ and $j=J+1,\dots,J+p$ will be referred to as ``ghost cells''. They are depicted in 
red in Figure \ref{fig:maillage} (in that example, $p=r=2$).

\begin{figure}[h!]
\begin{center}
\begin{tikzpicture}[scale=2,>=latex]
\draw [ultra thin, dotted, fill=blue!20] (0,0) rectangle (5,1.5);
\draw [thin, dashed, fill=white] (0,0) grid [step=0.5] (5,1.5);
\draw [ultra thin, dotted, fill=red!20] (5,0) rectangle (6,1.5);
\draw [thin, dashed, fill=white] (5,0) grid [step=0.5] (6,1.5);
\draw [ultra thin, dotted, fill=red!20] (-1,0) rectangle (0,1.5);
\draw [thin, dashed, fill=white] (-1,0) grid [step=0.5] (0,1.5);
\draw[black,->] (-1.5,0) -- (6.5,0) node[below] {$x$};
\draw[black,->] (0,0)--(0,2) node[right] {$t$};
\draw (0,0.5) node[left, fill=red!20]{$t^1$};
\draw (0,1) node[left, fill=red!20]{$t^2$};
\draw (0,1.5) node[left, fill=red!20]{$t^3$};
\draw (-1,0) node[below]{$x_{-2}$};
\draw (-0.5,0) node[below]{$x_{-1}$};
\draw (0,0) node[below]{$x_0$};
\draw (0,-0.3) node {$0$};
\draw (0.5,0) node[below]{$x_1$};
\draw (1,0) node[below]{$x_2$};
\draw (2.5,0) node[below]{$\cdots$};
\draw (4,0) node[below]{$x_{J-2}$};
\draw (4.5,0) node[below]{$x_{J-1}$};
\draw (5,0) node[below]{$x_J$};
\draw (5,-0.3) node {$L$};
\draw (5.5,0) node[below]{$x_{J+1}$};
\draw (6,0) node[below]{$x_{J+2}$};
\node (centre) at (0,0){$\bullet$};
\node (centre) at (5,0){$\bullet$};
\end{tikzpicture}
\caption{The mesh on $\R^+ \times (0,L)$ in blue, and the ``ghost cells'' in red ($r=p=2$ here).}
\label{fig:maillage}
\end{center}
\end{figure}
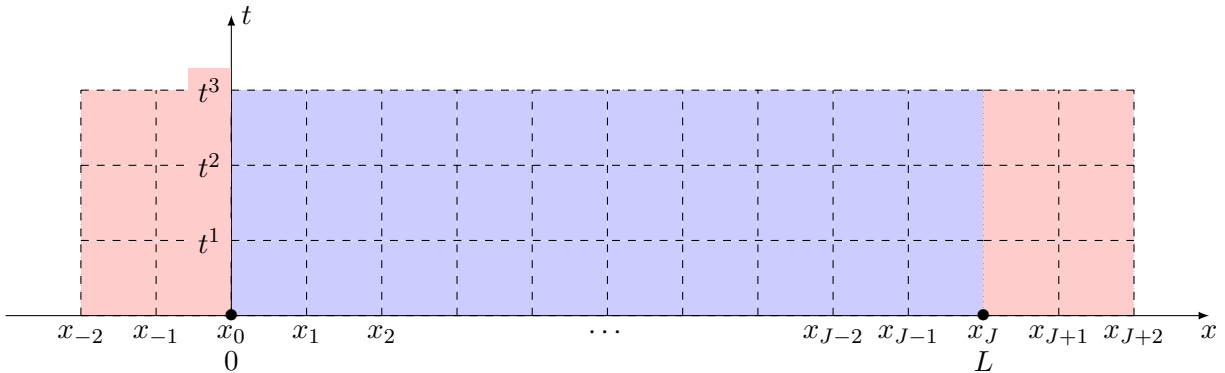

Before describing the numerical boundary conditions we enforce for \eqref{transportnum}, let us state our main and in fact only assumption 
on the coefficients in \eqref{transportnum}.

\begin{assumption}[Consistency and stability]
\label{as:scheme}
The coefficients $a_{-r},\dots,a_p$ in \eqref{transportnum} satisfy $a_{-r} \, a_p \neq 0$ (normalization), and for some integer $k \ge 1$, 
there holds:
\begin{align}
&\forall \, m=0,\dots,k \, ,\quad \sum_{\ell=-r}^p \ell^m \, a_\ell = (-\lambda \, a)^m \, ,\quad 
&\text{\rm (consistency of order $k$)} \, ,\label{eq:consist}\\
&\sup_{\theta \in [0,2\, \pi]} \left| \sum_{\ell=-r}^p a_\ell \, {\rm e}^{i \, \ell \, \theta} \right| \le 1 \, ,\quad 
&\text{\rm ($\ell^2$-stability)} \, .\label{eq:stab}
\end{align}
\end{assumption}

Provided that the relations \eqref{eq:consist} are satisfied for $m=0$ (conservativity) and $m=1$ (consistency of order $1$) with $a>0$, the 
stability assumption \eqref{eq:stab} implies $r \ge 1$, which we assume from now on. Though we view this observation here, as in \cite{strang1}, 
as a \emph{necessary} condition for \emph{stability}, the condition $r \ge 1$ is also known to be \emph{necessary} for \emph{convergence} 
by comparing the numerical and continuous dependency domains, see \cite{cfl}. Let us observe that \eqref{eq:stab} is a necessary and 
sufficient condition for stability of the iteration process \eqref{transportnum} on $\ell^2(\Z)$ in a strong sense, meaning here that the map
$$
(v_j)_{j \in \Z} \longmapsto \left( \sum_{\ell=-r}^p a_\ell \, v_{j+\ell} \right)_{j\in \Z}
$$
is a contraction (it has norm $1$) as an operator on $\ell^2(\Z)$. However, \eqref{eq:stab} is not sufficient to yield stability in $\ell^\infty(\Z)$ for 
\eqref{transportnum}, see \cite{thomee,hedstrom}. Note that through the dependence of the $a_\ell$ with respect to $\lambda = \Delta t/\Delta x$, 
\eqref{eq:stab} is usually intended to be true only under a so-called Courant-Friedrichs-Lewy condition asking for $\lambda$ to be less than some 
constant depending on the scheme and the velocity $a$. (Indeed, the Bernstein inequality for trigonometric polynomials implies $\lambda \, |a| \le 
\max (p,r)$, see \cite{strang1}.) We postpone the extension of our work to $\ell^\infty(\Z)$ stable discretizations to a future work since the methods 
to be used in that framework will definitely need to be different from the more ``Hilbertian'' techniques we use here.

Let us observe that Assumption \ref{as:scheme} does not include any \emph{dissipative} behavior for 
\eqref{transportnum}, meaning that we do not assume a bound of the form:
$$
\forall \, \theta \in [-\pi,\pi] \, ,\quad \left| \sum_{\ell=-r}^p a_\ell \, {\rm e}^{i \, \ell \, \theta} \right| 
\le 1 -c\, \theta^{2\, q} \, ,
$$
for some suitable integer $q$ and positive constant $c$. In that respect, the framework of Assumption 
\ref{as:scheme} is more general than the works \cite{kreiss1,goldberg,gks,osher1} and following works 
that are based on these pioneering results. We thus expect that our approach may be useful to deal 
with multidimensional problems in which dissipativity is most of the time excluded (or restrictive).

If the interior cells of the grid are labeled, as in \eqref{transportnum}, by $j \in \{ 1,\dots,J \}$, the numerical 
approximation of \eqref{transport} requires, for passing from one time index $n$ to the next, prescribing 
$r$ numerical boundary conditions on the left of the interval (that is, close to $x=0$), and $p$ numerical 
boundary conditions on the right (that is, close to $x=L$). In other words, we need to prescribe the value 
of the approximate solution $(u_j^n)$ in the ghost cells located at the boundary of the interior domain. 
For simplicity, and in order to be consistent with the continuous problem \eqref{transport}, we prescribe 
Dirichlet homogeneous boundary conditions in conjunction with \eqref{transportnum} on the left of the 
interval $(0,L)$:
\begin{equation}
\label{dirichlet}
u_\ell^{n} =0 \, ,\quad \ell=1-r,\dots,0 \, .
\end{equation}
On the right of the interval $(0,L)$, there is nothing to be done if $p=0$, that is, in the case of an \emph{upwind} 
discretization, for in that case, given the vector $(u_{1}^{n},\dots,u_J^{n})$, the vector $(u_{1}^{n+1},\dots,u_J^{n+1})$ 
is entirely determined by \eqref{dirichlet} and \eqref{transportnum}, so we can iterate the scheme 
\eqref{transportnum}-\eqref{dirichlet} starting from some initial data $(u_{1}^0,\dots,u_J^0)$ to any positive time level 
$n$. We therefore assume from now on $p \ge 1$, which is the interesting case where the numerical discretization of 
\eqref{transport} necessitates an outflow numerical boundary condition while the continuous problem does not ``obviously'' 
provide with one. In this article, we shall prescribe Neumann type numerical boundary conditions (these are called 
\emph{extrapolation} numerical boundary conditions in \cite{goldberg}). For ease of writing, we introduce the difference 
operator in space which acts on vectors $(v_j)_{j=1-r,\dots,J+p}$ as follows:
$$
\forall \, j=2-r,\dots,J+p \, ,\quad (D_-v)_j :=v_j -v_{j-1} \, .
$$
Higher order difference operators $D_-^m$, $m \ge 2$, are defined accordingly by iterating $D_-$. Then given 
a fixed integer $k_b \in \N$ ($b$ stands for  ``boundary''), we prescribe the following numerical boundary condition 
in conjunction with \eqref{transportnum}:
\begin{equation}
\label{neumann}
(D_-^{k_b} u^{n})_{J+\ell} =0 \, ,\quad \ell=1,\dots,p \, .
\end{equation}
If $k_b=0$, this corresponds to prescribing homogeneous Dirichlet boundary conditions, while if $k_b=1$, 
this corresponds to the standard Neumann numerical boundary condition:
$$
u^{n}_{J+1} =\cdots=u^{n}_{J+p} :=u^{n}_J \, .
$$
The iteration \eqref{transportnum}, \eqref{dirichlet}, \eqref{neumann} thus proceeds as follows, see Figure 
\ref{fig:schema} for an illustration. Given the vector $(u_{1}^n,\dots,u_{J}^n)$ for some time level $n$, one first 
determines the ghost values $(u_{1-r}^{n},\dots,u_0^{n},u_{J+1}^{n},\dots,u_{J+p}^{n})$ by \eqref{dirichlet} and 
\eqref{neumann}. The \emph{new} vector $(u_1^{n+1},\dots,u_J^{n+1})$ is then determined by applying 
\eqref{transportnum}. It is assumed that $J \geq 1$ in order to make the space step $\Delta x=L/J$ meaningful 
and to have at least one cell in the interval $(0,L)$.

\begin{figure}[H]
\begin{center}
\begin{tikzpicture}[scale=2,>=latex]
\draw [ultra thin, dotted, fill=gray!20] (-1,0) rectangle (6,0.5);
\draw [thin, dashed] (-1,0) grid [step=0.5] (6,0.5);
\draw [ultra thin, dotted, fill=gray!20] (0,0.5) rectangle (5,1);
\draw [thin, dashed] (0,0.5) grid [step=0.5] (5,1);
\draw [ultra thin, dotted, fill=red!20] (-1,0) rectangle (0,0.5);
\draw [thin, dashed] (-1,0.5) grid [step=0.5] (0,1);
\draw [thin, dashed] (-1,0) grid [step=0.5] (0,0.5);
\draw [ultra thin, dotted, fill=red!20] (5,0) rectangle (6,0.5);
\draw [thin, dashed] (5,0.5) grid [step=0.5] (6,1);
\draw [thin, dashed] (5,0) grid [step=0.5] (6,0.5);
\draw [thin,<-] (5.3,0.3) arc (0:180:0.55);
\draw [thin,<-] (5.25,0.3) arc (0:180:0.25);
\draw [thin,->] (4.7,0.2) arc (180:360:0.55);
\draw [thin,->] (5.25,0.2) arc (180:360:0.25);
\draw[black,->] (-1.5,0) -- (6.5,0) node[below] {$x$};
\draw[black,->] (0,0)--(0,1.5) node[right] {$t$};
\draw (-1,0.5) node[left]{$t^1$};
\draw (-1,1) node[left]{$t^2$};
\draw (-1,0) node[below]{$x_{-2}$};
\draw (-0.5,0) node[below]{$x_{-1}$};
\draw (0,0) node[below]{$x_0$};
\draw (0.5,0) node[below]{$x_1$};
\draw (1,0) node[below]{$x_2$};
\draw (2.5,0) node[below]{$\cdots$};
\draw (4,0) node[below]{$x_{J-2}$};
\draw (4.5,0) node[below]{$x_{J-1}$};
\draw (5,0) node[below]{$x_J$};
\draw (5.5,0) node[below]{$x_{J+1}$};
\draw (6,0) node[below]{$x_{J+2}$};
\node (centre) at (0,0){$\bullet$};
\node (centre) at (5,0){$\bullet$};
\draw [ultra thin, dotted, fill=gray!20] (-1,-2.5) rectangle (6,-2);
\draw [thin, dashed, fill=white] (-1,-2.5) grid [step=0.5] (6,-2);
\draw [ultra thin, dotted, fill=blue!20] (0,-2) rectangle (5,-1.5);
\draw [thin, dashed, fill=white] (0,-2) grid [step=0.5] (5,-1.5);
\draw[black,->] (-0.75,-2.25) -- (0.15,-1.75);
\draw[black,->] (-0.25,-2.25) -- (0.2,-1.75);
\draw[black,->] (0.25,-2.25) -- (0.25,-1.75);
\draw[black,->] (0.75,-2.25) -- (0.3,-1.75);
\draw[black,->] (1.25,-2.25) -- (0.35,-1.75);
\draw[black,->] (3.75,-2.25) -- (4.65,-1.75);
\draw[black,->] (4.25,-2.25) -- (4.7,-1.75);
\draw[black,->] (4.75,-2.25) -- (4.75,-1.75);
\draw[black,->] (5.25,-2.25) -- (4.8,-1.75);
\draw[black,->] (5.75,-2.25) -- (4.85,-1.75);
\draw[black,->] (-1.5,-2.5) -- (6.5,-2.5) node[below] {$x$};
\draw[black,->] (0,-2.5)--(0,-1) node[right] {$t$};
\draw (-1,-2) node[left]{$t^1$};
\draw (-1,-1.5) node[left]{$t^2$};
\draw (-1,-2.5) node[below]{$x_{-2}$};
\draw (-0.5,-2.5) node[below]{$x_{-1}$};
\draw (0,-2.5) node[below]{$x_0$};
\draw (0.5,-2.5) node[below]{$x_1$};
\draw (1,-2.5) node[below]{$x_2$};
\draw (2.5,-2.5) node[below]{$\cdots$};
\draw (4,-2.5) node[below]{$x_{J-2}$};
\draw (4.5,-2.5) node[below]{$x_{J-1}$};
\draw (5,-2.5) node[below]{$x_J$};
\draw (5.5,-2.5) node[below]{$x_{J+1}$};
\draw (6,-2.5) node[below]{$x_{J+2}$};
\node (centre) at (0,-2.5){$\bullet$};
\node (centre) at (5,-2.5){$\bullet$};
\end{tikzpicture}
\caption{Top: updating iteratively the ghost values at the outflow boundary ($r=p=k_b=2$). Bottom: updating the 
numerical approximation in the interior.}
\label{fig:schema}
\end{center}
\end{figure}
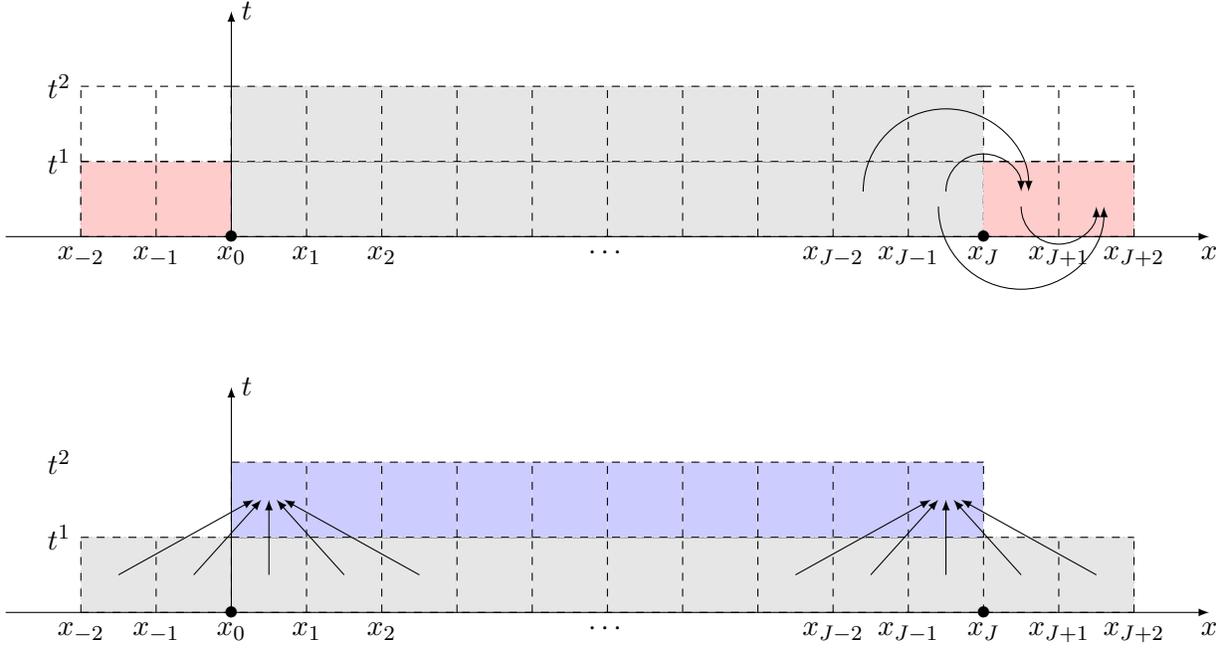

Let us emphasize that the same procedure \eqref{neumann} is applied at each ghost cell close to the outflow 
boundary $x=L$. In the terminology of \cite{goldberg-tadmor1,goldberg-tadmor}, the numerical boundary 
conditions are of \emph{translatory} type.

The scheme \eqref{transportnum}, \eqref{dirichlet}, \eqref{neumann} is initialized with the piecewise constant 
projection of the initial condition for \eqref{transport}, that is, for the interior cells:
\begin{equation}
\label{initial1}
\forall \, j=1,\dots,J \, ,\quad u_j^0 :=\dfrac{1}{\Delta x} \, \int_{(j-1) \, \Delta x}^{j \, \Delta x} u_0(y) \, {\rm d}y \, .
\end{equation}
Our main result is the following convergence estimate for the scheme \eqref{transportnum}, \eqref{dirichlet}, 
\eqref{neumann} supplemented with the initial condition \eqref{initial1}.

\begin{theorem}
\label{thm1}
Let $a>0$, let $k \in \N^*$ and $k_b \in \N$, let $\lambda > 0$ and assume that the coefficients in \eqref{transportnum} 
satisfy Assumption \ref{as:scheme} with integer $k$. Then there exists $C>0$ such that for any $T > 0$ and $J \in \N^*$, 
for any $L \geq 1$, and for any $u_0 \in H^{k+1}((0,L))$ satisfying the compatibility requirements at $x=0$:
$$
\forall \, m=0,\dots,k \, ,\quad u_0^{(m)}(0)=0 \, , 
$$
the solution $(u_j^n)$ to \eqref{transportnum}, \eqref{dirichlet}, \eqref{neumann} with initial datum \eqref{initial1} satisfies:
\begin{equation}
\label{estimlinfty}
\sup_{0 \leq n \leq T/\Delta t} \sup_{j=1,\dots,J} \left| u_j^n -\dfrac{1}{\Delta x} \, \int_{x_{j-1}}^{x_j} u_0(y-a\, t^n) \, {\rm d}y \right| 
\le C \, \big( \sqrt{T}+T \big) \, {\rm e}^{C\, T/L} \, \Delta x^{\min (k,k_b) -1/2} \, \| u_0 \|_{H^{k+1}((0,L))} \, ,
\end{equation}
where $\Delta x = L/J$ and $\Delta t = \lambda \, \Delta x$. It is understood in \eqref{estimlinfty} that $u_0$ is extended by 
$0$ to $\R^-$.
\end{theorem}

The restriction to numerical schemes with two time levels only in \eqref{transportnum} is necessary since, for instance, the 
Neumann numerical boundary condition \eqref{neumann} is known to yield violent instabilities when used in conjunction with 
the leap-frog scheme \cite{trefethenCPAM}. We postpone the study of outflow numerical boundary conditions for general 
multistep schemes to a future work.

The integer $k_b$ in \eqref{neumann} prescribes the approximation order of the numerical outflow boundary condition. In particular, 
Theorem \ref{thm1} shows that the Neumann numerical boundary condition ($k_b=1$) gives a local, stable hence convergent way 
to approximate the exact discrete transparent boundary condition for \eqref{transportnum} which is nonlocal in time (see \cite{jfcAFST} 
for a general derivation of discrete transparent boundary conditions).

There is a ``loss'' of $1/2$ in the somehow expected rate of convergence $\min (k,k_b)$ in \eqref{estimlinfty}. 
We emphasize once again that Assumption \ref{as:scheme} does not imply stability of \eqref{transportnum} in 
$\ell^\infty(\Z)$ and therefore even on the whole real line, the rate of convergence $k$ in $\ell^\infty$ cannot be 
attained in general. With the additional technical complexity of dealing with numerical boundary conditions, we 
view this loss of $1/2$ as a minor disagreement of our method.

Applying estimate \eqref{estimlinfty} to the ``theoretical'' final time $T:=L/a$ after which the exact solution to the 
transport equation \eqref{transport} becomes zero (everything has flowed out of the interval $(0,L)$ through the 
outflow boundary $x=L$), we obtain:
$$
\sup_{j=1,\dots,J} \left| u_j^{N+1} \right| \lesssim \Delta x^{\min (k,k_b) -1/2} \, ,\quad N :={\rm E} \left( \dfrac{L}{a \, \Delta t} \right) \, ,
$$
meaning that the numerical approximation of the solution to \eqref{transport} has become uniformly small 
on the interval $(0,L)$ at time level $N+1$. At later times, stability estimates for the numerical scheme 
\eqref{transportnum}, \eqref{dirichlet}, \eqref{neumann}, which we shall prove below, assert that the solution 
$(u_j^n)$ to \eqref{transportnum}, \eqref{dirichlet}, \eqref{neumann} remains ``small'' (at least on any given 
finite time interval, $J$ being large).

\section{Preliminary analysis on a half-line}
\label{sect:3}

In this section, we show a preliminary result, which is Theorem \ref{thm3} below, that is entirely analogous 
to our main result, Theorem \ref{thm1} above, except that the space domain is a half-line with an outgoing 
velocity at the boundary. The result is interesting on its own since a byproduct of our analysis is the verification 
of the so-called Uniform Kreiss Lopatinskii Condition though we completely bypass the normal mode analysis 
that is commonly used to verify it \cite{goldberg}. The convergence analysis on the half-line is more classical. 
In addition to its own interest, which somehow focuses on the Neumann boundary condition without coupling 
it to Dirichlet boundary condition, Theorem \ref{thm3} below will be a major building block in the proof of 
Theorem \ref{thm1}.

\subsection{Stability estimates for Neumann numerical outflow boundary conditions}

In this paragraph, we prove Theorem \ref{thm2} below that provides us with stability estimates 
for the Neumann numerical boundary condition on a half-line. Theorem \ref{thm2} is a key tool for proving 
stability estimates for the scheme \eqref{transportnum}, \eqref{dirichlet}, \eqref{neumann} on a finite interval, 
which in turn yields the convergence result of Theorem \ref{thm1}. Let us recall that Theorem \ref{thm2} 
below is already known to hold true thanks to the joint results of \cite{kreissproc,goldberg,kreiss1} and 
in a more general setting \cite{wu,jfcag,jfcnotes} (see also references therein). It roughly means that the 
Neumann numerical boundary conditions satisfy the Uniform Kreiss Lopatinskii Condition at an outflow 
boundary. However we emphasize that Assumption \ref{as:scheme} does not imply ``dissipativity'' for 
\eqref{transportnum}, which already prohibits using the main results in \cite{kreiss1,gks} and therefore  
strongly advocates using the energy method, as we shall do, whenever possible. Independently of the 
subtleties of the assumptions in those works, we give here an elementary proof of the stability result in 
\cite{kreissproc,goldberg} without any Laplace transform nor ``GKS'' type arguments, and with more easily 
checkable assumptions. Our integration by parts approach was probably already well-known for three point 
schemes and the classical Neumann boundary condition $D_-u=0$, but our main contribution is to show that 
the energy method can be used in order to deal with numerical schemes with an arbitrarily wide stencil and 
for any extrapolation order at the boundary.

Before going on, let us fix the space domain that we consider. Since we deal with a constant coefficient 
linear problem, by translation invariance, there is no loss of generality in considering the half-line $(-\infty,L)$. 
With the same positive constant $\lambda>0$ as in the previous section, we consider $J \in \N^*$, and some 
space and time steps $\Delta x = L/J$ and $\Delta t = \lambda \, \Delta x$, assuming that $J$ is large enough 
to ensure that $\Delta t \in (0,1]$. We also keep the notation $t^n:=n\, \Delta t$, $n \in \N$ and $x_j:=j\, \Delta x$, 
$j \in \Z$. The grid and the associated ghost cells are depicted in Figure \ref{fig:maillage'}. Our result can be stated 
as follows.

\begin{figure}[h!]
\begin{center}
\begin{tikzpicture}[scale=2,>=latex]
\draw [ultra thin, dotted, fill=blue!20] (0,0) rectangle (5.2,1.2);
\draw [thin, dashed, fill=white] (0,0) grid [step=0.4] (5.2,1.2);
\draw [ultra thin, dotted, fill=red!20] (5.2,0) rectangle (6,1.2);
\draw [thin, dashed, fill=white] (5.2,0) grid [step=0.4] (6,1.2);
\draw [ultra thin, dotted, fill=blue!20] (-1.4,0) rectangle (0,1.2);
\draw [thin, dashed, fill=white] (-1.4,0) grid [step=0.4] (0,1.2);
\draw[black,->] (-1.5,0) -- (6.5,0) node[below] {$x$};
\draw[black,->] (0.4,0)--(0.4,2) node[right] {$t$};
\draw (0,0.4) node[left, fill=blue!20]{$t^1$};
\draw (0,0.8) node[left, fill=blue!20]{$t^2$};
\draw (-1.2,0) node[below]{$\cdots$};
\draw (-0.8,0) node[below]{$x_{-3}$};
\draw (-0.4,0) node[below]{$x_{-2}$};
\draw (0,0) node[below]{$x_{-1}$};
\draw (0.4,0) node[below]{$x_0$};
\draw (0.8,0) node[below]{$x_1$};
\draw (2.2,0) node[below]{$\cdots$};
\draw (4.4,0) node[below]{$x_{J-2}$};
\draw (4.8,0) node[below]{$x_{J-1}$};
\draw (5.2,0) node[below]{$x_J$};
\draw (5.2,-0.3) node {$L$};
\draw (5.6,0) node[below]{$x_{J+1}$};
\draw (6,0) node[below]{$x_{J+2}$};
\node (centre) at (5.2,0){$\bullet$};
\end{tikzpicture}
\caption{The mesh on $\R^+ \times (-\infty,L)$ in blue, and the ``ghost cells'' in red ($p=2$ here).}
\label{fig:maillage'}
\end{center}
\end{figure}

\begin{theorem}
\label{thm2}
Let $a>0$, let $k\ge 1$ and $k_b \in \N$, let $\lambda > 0$. Then there exists a constant $C>0$ such that for all initial data 
$(f_j)_{j \le J} \in \ell^2$ and for all boundary source terms $(g_{J+1}^n)_{n \ge 0},\dots,(g_{J+p}^n)_{n \ge 0}$ verifying the 
growth condition:
$$
\forall \, \Gamma >0 \, ,\quad \sum_{n \ge 0} {\rm e}^{-2\, \Gamma \, n} \, \Big( (g_{J+1}^n)^2 +\cdots +(g_{J+p}^n)^2 \Big) < +\infty \, ,
$$
the solution $(u_j^n)_{j \le J+p,n\in \N}$ to the scheme:
\begin{equation}
\label{pbauxk}
\begin{cases}
u_j^0=f_j \, ,& j\le J\, ,  \\
(D_-^{k_b} u^{n})_{J+\ell} =g_{J+\ell}^{n} \, ,& n \in \N \, ,\quad \ell=1,\dots,p \, ,\\
u_j^{n+1} ={\displaystyle \sum_{\ell=-r}^p} a_\ell \, u_{j+\ell}^n \, ,& n \in \N \, ,\quad j \le J \, ,\\
\end{cases}
\end{equation}
where the coefficients $a_\ell$ satisfy Assumption \ref{as:scheme} with integer $k$, satisfies the estimate:
\begin{multline}
\label{estimation}
\sup_{n \in \N} \, \left({\rm e}^{-2\, \gamma \, n\, \Delta t} \, \sum_{j \le J} \, \Delta x \, (u_j^n)^2 \right)
+\sum_{n\ge 0} \, \Delta t \, {\rm e}^{-2\, \gamma \, n\, \Delta t} \, \sum_{\ell=1-r-k_b}^p (u_{J+\ell}^n)^2 \\
\le C \, \left\{ \sum_{j \le J} \, \Delta x \, (f_j)^2 
+\sum_{n\ge 0} \, \Delta t \, {\rm e}^{-2\, \gamma \, n\, \Delta t} \, \sum_{\ell=1}^p (g_{J+\ell}^n)^2 \right\} 
\end{multline}
for any $\gamma > 0$, and for any $\Delta t$ and $\Delta x$ such that $\Delta t/\Delta x = \lambda$ and $\Delta t \in (0, 1]$. 
In particular, the numerical boundary conditions in \eqref{pbauxk} satisfy the Uniform Kreiss Lopatinskii Condition.
\end{theorem}

Before proving Theorem \ref{thm2}, let us recall that we always assume the ratio $\Delta t/\Delta x$ to be constant. This 
will be used several times below and is reminiscent of the scale invariance properties of the underlying continuous problem. 
Observe also that in \eqref{estimation}, the larger the integer $k_b$ the better the trace estimate on the left hand side behaves. 
In particular, the numerical boundary conditions in \eqref{pbauxk} involve the values $u_{J+1-k_b}^{n},\dots,u_{J+p}^{n}$, and 
we get ``for free'' in \eqref{estimation} not only the control of those terms but also the extra control of $u_{J+1-r-k_b}^{n}, \dots, 
u_{J-k_b}^{n}$ (recall $r \ge 1$). The fact that \eqref{estimation} implies the Uniform Kreiss Lopatinskii Condition is not our main 
focus here, so instead of recalling many definitions, we rather refer the interested reader to the review \cite{jfcnotes}.

\begin{proof}
We shall use Assumption \ref{as:scheme} in the proof below only for $k=1$, that is, we make the ``minimal'' 
consistency requirements for the scheme \eqref{transportnum}. Unlike \cite{kreissproc,goldberg}, the proof 
of Theorem \ref{thm2} is done by \emph{induction} with respect to the index $k_b \in \N$ and relies on the 
\emph{energy method}. Let us start with the case $k_b=0$, which corresponds to Dirichlet numerical boundary 
conditions.
\bigskip

$\bullet$ \underline{The case $k_b=0$}. We consider the numerical scheme:
\begin{equation}
\label{pbauxDir}
\begin{cases}
u_j^0=f_j \, ,& j\le J \, , \\
u^{n}_{J+\ell} =g_{J+\ell}^{n} \, ,& n \in \N \, ,\quad \ell=1,\dots,p \, ,\\
u_j^{n+1} ={\displaystyle \sum_{\ell=-r}^p} a_\ell \, u_{j+\ell}^n \, ,& n \in \N \, ,\quad j \le J \, .\\
\end{cases}
\end{equation}
A straightforward proof of the stability estimate \eqref{estimation} for $k_b=0$ was achieved in \cite{jfcag} 
(even in some cases of multidimensional systems), see also \cite{goldberg-tadmor} for an earlier general 
result based on the theory of \cite{gks}. We reproduce here the short proof of \eqref{estimation} for the 
scheme \eqref{pbauxDir} for the sake of completeness.

Assumption \ref{as:scheme} implies by the Plancherel Theorem that the mapping:
\begin{equation}
\label{contraction}
(v_j)_{j \in \Z} \in \ell^2 \longmapsto \left( \sum_{\ell=-r}^p a_\ell \, v_{j+\ell} \right)_{j \in \Z} \in \ell^2 \, ,
\end{equation}
is a contraction, in the sense that the operator norm is not larger than $1$. Let us now consider the solution 
$(u_j^n)_{j \le J+p,n\in \N}$ to \eqref{pbauxDir} at some time index $n \in \N$. We extend the sequence 
$(u_j^n)_{j \le J+p}$ by $0$ for $j \ge J+p+1$ and still denote $u^n \in \ell^2(\Z)$ the resulting sequence. 
Let us then define
$$
\forall \, j \in \Z \, ,\quad v_j^{n+1} := \sum_{\ell=-r}^p a_\ell \, u_{j+\ell}^n \, ,
$$
so that $v_j^{n+1}=u_j^{n+1}$ for $j \le J$, and $v_j^{n+1}=0$ if $j \ge J+p+r+1$. Observe that, due to the boundary 
conditions in \eqref{pbauxDir}, we do not necessarily have $v_j^{n+1}=u_j^{n+1}$ for $j=J+1,\dots,J+p$, nor for 
$j=J+p+1,\dots,J+p+r$ (extending also $(u_j^{n+1})_{j \le J+p}$ by $0$ for $j \ge J+p+1$). Using Assumption 
\ref{as:scheme}, we get:
$$
\sum_{j \le J} \, \Delta x \, (u_j^{n+1})^2 +\sum_{j=J+1}^{J+p+r} \, \Delta x \, (v_j^{n+1})^2 
=\sum_{j \in \Z} \, \Delta x \, (v_j^{n+1})^2 \le \sum_{j \in \Z} \, \Delta x \, (u_j^n)^2 
=\sum_{j \le J+p} \, \Delta x \, (u_j^n)^2 \, ,
$$
since \eqref{contraction} is a contraction. Equivalently, taking the boundary conditions of \eqref{pbauxDir} into account, 
we get:
\begin{equation}
\label{estimthm2-1}
\underbrace{\sum_{j \le J} \, \Delta x \, (u_j^{n+1})^2 
-\sum_{j \le J} \, \Delta x \, (u_j^n)^2}_{\text{\rm Discrete time derivative}} 
+\underbrace{\Delta x \, \sum_{\ell=1}^{p+r} \, (v_{J+\ell}^{n+1})^2}_{\text{\rm Trace term}} 
\le \underbrace{\Delta x \, \sum_{\ell=1}^p \, (g_{J+\ell}^n)^2}_{\text{\rm Source term}}  \, .
\end{equation}
We now derive a bound from below for the trace term arising on the left hand side of \eqref{estimthm2-1}. 
The real numbers $v_{J+\ell}^{n+1}$, $\ell=1,\dots,p+r$, depend linearly on $u_{J+1-r}^n,\dots,u_{J+p}^n$. 
The coefficients in each linear combination are taken among the $a_\ell$'s. Hence the quantity
$$
\sum_{\ell=1}^{p+r} \, (v_{J+\ell}^{n+1})^2 
$$
can be seen as a nonnegative quadratic form in the variables $u_{J+1-r}^n,\dots,u_{J+p}^n$. It is also rather 
easy to see that this quadratic form is \emph{positive definite} for we have $v_{J+p+r}^{n+1}=a_{-r} \, u_{J+p}^n$, 
and, therefore, if $v_{J+1}^{n+1} =\cdots =v_{J+p+r}^{n+1} =0$, then we first have $u_{J+p}^n=0$ and 
recursively we can also show $u_{J+p-1}^n =\cdots =u_{J+1-r}^n =0$. Hence there exists a fixed constant 
$c_0>0$, that only depends on the (fixed) coefficients $a_\ell$ in \eqref{transportnum}, such that:
$$
\sum_{\ell=1}^{p+r} \, (v_{J+\ell}^{n+1})^2 \ge c_0 \, \sum_{\ell=1-r}^p \, (u_{J+\ell}^n)^2 \, .
$$
Reporting in \eqref{estimthm2-1} and using that $\Delta t/\Delta x =\lambda$ is a fixed positive constant, 
we get for some constant $c>0$ ($c = c_0/\lambda$ is suitable):
\begin{equation}
\label{estimthm2-2}
\sum_{j \le J} \, \Delta x \, (u_j^{n+1})^2 -\sum_{j \le J} \, \Delta x \, (u_j^n)^2 
+c \, \Delta t \, \sum_{\ell=1-r}^p \, (u_{J+\ell}^n)^2 \le \dfrac{1}{\lambda} \, \Delta t \, 
\sum_{\ell=1}^p \, (g_{J+\ell}^n)^2 \, .
\end{equation}
We now apply the following discrete Gronwall type lemma (with the positive parameter $\Gamma := \gamma 
\, \Delta t$), see \cite{jfcag} for repeated use of such summation arguments.

\begin{lemma}
\label{gronwall}
Let $({\mathcal G}_n)_{n \ge 0}$ be a sequence of nonnegative real numbers such that:
$$
\forall \, \Gamma >0 \, ,\quad \sum_{n \ge 0} {\rm e}^{-2\, \Gamma \, n} \, {\mathcal G}_n <+\infty \, .
$$
Let $({\mathcal U}_n)_{n \ge 0}$, $({\mathcal B}_n)_{n \ge 0}$ be two sequences of nonnegative real 
numbers such that:
$$
\forall \, n \in \N \, ,\quad {\mathcal U}_{n+1} -{\mathcal U}_n +{\mathcal B}_n \le {\mathcal G}_n \, .
$$
Then there holds for all $\Gamma>0$:
$$
\sup_{n \in \N} \,  {\rm e}^{-2\, \Gamma \, n} \, {\mathcal U}_n 
+\sum_{n \ge 0} \, {\rm e}^{-2\, \Gamma \, n} \, {\mathcal B}_n \le {\mathcal U}_0 
+\sum_{n \ge 0} \, {\rm e}^{-2\, \Gamma \, n} \, {\mathcal G}_n \, .
$$
\end{lemma}

\noindent The proof of Lemma \ref{gronwall} is straightforward and therefore omitted. We apply Lemma 
\ref{gronwall} to \eqref{estimthm2-2} and derive the estimate:
\begin{multline}
\label{estimthm2-2'}
\sup_{n \in \N} \, \left( {\rm e}^{-2\, \gamma \, n \, \Delta t} \, \sum_{j \le J} \, \Delta x \, (u_j^n)^2 \right) 
+\sum_{n \ge 0} \, \Delta t \, {\rm e}^{-2\, \gamma \, n \, \Delta t} \, \sum_{\ell=1-r}^p \, (u_{J+\ell}^n)^2 \\
\le C \, \left\{ \sum_{j \le J} \, \Delta x \, (f_j)^2 
+\sum_{n \ge 0} \, \Delta t \, {\rm e}^{-2\, \gamma \, n \, \Delta t} \, \sum_{\ell=1}^p \, (g_{J+\ell}^n)^2 \right\} \, ,
\end{multline}
which is \eqref{estimation} for $k_b=0$.

We emphasize that when dealing with the case $k_b=0$, we have only used the stability condition \eqref{eq:stab} 
of Assumption \ref{as:scheme}, and we have never used \eqref{eq:consist} (not even for $m=0$). This is consistent 
with the result of \cite{goldberg-tadmor} which proves that the Dirichlet boundary condition satisfies the Uniform Kreiss 
Lopatinskii Condition independently of the nature of the boundary (inflow or outflow).
\bigskip

$\bullet$ \underline{The induction argument}. We now assume that the stability estimate \eqref{estimation} 
is valid up to some index $k_b \in \N$, and consider the following numerical scheme:
\begin{equation}
\label{pbauxk+1}
\begin{cases}
u_j^0=f_j \, ,& j\le J \, , \\
(D_-^{k_b+1} u^{n})_{J+\ell} =g_{J+\ell}^{n} \, ,& n \in \N \, ,\quad \ell=1,\dots,p \, ,\\
u_j^{n+1} ={\displaystyle \sum_{\ell=-r}^p} a_\ell \, u_{j+\ell}^n \, ,& n \in \N \, ,\quad j \le J \, .\\
\end{cases}
\end{equation}
We use the induction assumption by defining the following sequence:
$$
\forall \, n \in \N \, ,\quad \forall \, j \le J+p \, ,\quad w_j^n :=(D_- \, u^n)_j =u_j^n -u_{j-1}^n\, ,
$$
which satisfies:
\begin{equation*}
\begin{cases}
w_j^0=f_j-f_{j-1} \, ,& j\le J \, , \\
(D_-^{k_b} w^{n})_{J+\ell} =g_{J+\ell}^{n} \, ,& n \in \N \, ,\quad \ell=1,\dots,p \, ,\\
w_j^{n+1} ={\displaystyle \sum_{\ell=-r}^p} a_\ell \, w_{j+\ell}^n \, ,& n \in \N \, ,\quad j \le J \, .\\
\end{cases}
\end{equation*}
Applying the stability estimate \eqref{estimation} for the index $k_b$ and omitting one of the two nonnegative 
terms on the left hand side of \eqref{estimation}, we have already derived the preliminary estimate:
\begin{align}
\sum_{n\ge 0} \, \Delta t \, {\rm e}^{-2\, \gamma \, n\, \Delta t} \, 
\sum_{\ell=1-r-k_b}^p (u_{J+\ell}^n-u_{J+\ell-1}^n)^2 &\le C \left\{ \sum_{j \le J} \, \Delta x \, (f_j-f_{j-1})^2 
+\sum_{n\ge 0} \, \Delta t \, {\rm e}^{-2\, \gamma \, n\, \Delta t} \, \sum_{\ell=1}^p (g_{J+\ell}^n)^2 \right\} \notag \\
&\le C \left\{ \sum_{j \le J} \, \Delta x \, (f_j)^2 
+\sum_{n\ge 0} \, \Delta t \, {\rm e}^{-2\, \gamma \, n\, \Delta t} \, \sum_{\ell=1}^p (g_{J+\ell}^n)^2 \right\} 
\, .\label{estimthm2-3}
\end{align}

Let us now turn back to the numerical scheme \eqref{pbauxk+1} to which we are going to apply the 
so-called energy method. For a given time index $n \in \N$, we compute:
\begin{equation}
\label{energy1}
\sum_{j \le J} \, \Delta x \, (u_j^{n+1})^2 -\sum_{j \le J} \, \Delta x \, (u_j^n)^2 
=\sum_{j \le J} \Delta x \, \left\{ 2\, u_j^n \, \Big( \sum_{\ell=-r}^p a_\ell \, u_{j+\ell}^n -u_j^n \Big) 
+\Big( \sum_{\ell=-r}^p a_\ell \, u_{j+\ell}^n -u_j^n \Big)^2 \right\} \, .
\end{equation}
We use the following discrete integration by parts result, whose proof is postponed to Appendix \ref{appA}.

\begin{lemma}
\label{ippdiscrete}
Let $a>0$ and let Assumption \ref{as:scheme} be satisfied for some $k \in \N^*$. 
Then there exist a unique quadratic form 
${\mathcal Q}$ on $\R^{p+r}$ and some real coefficients $d_1,\dots,d_{p+r}$ such that, for any real 
numbers $v_{j-r},\dots,v_{j+p}$, there holds
\begin{align*}
2\, v_j \, \Big( \sum_{\ell=-r}^p a_\ell \, v_{j+\ell} -v_j \Big) &+\Big( \sum_{\ell=-r}^p a_\ell \, v_{j+\ell} -v_j \Big)^2 
= \sum_{\ell=1}^{p+r} \, d_\ell \, (v_{j+\ell-r} -v_{j-r})^2 \\
&\, +{\mathcal Q} \big( v_{j+2-r}-v_{j+1-r},\dots,v_j-v_{j-1},v_j,v_{j+1}-v_j,\dots,v_{j+p}-v_{j+p-1} \big) \\
&\, -{\mathcal Q} \big( v_{j+1-r}-v_{j-r},\dots,v_{j-1}-v_{j-2},v_{j-1},v_j-v_{j-1},\dots,v_{j+p-1}-v_{j+p-2} \big) \, .
\end{align*}
Furthermore, the quadratic form ${\mathcal Q}$ satisfies
$$
{\mathcal Q} (0,\dots,0,\underbrace{1}_{\text{\rm $r$-th entry}},0,\dots,0) 
= \sum_{\ell=-r}^p \ell \, a_\ell =-\lambda \, a <0 \, .
$$
\end{lemma}

We apply Lemma \ref{ippdiscrete} to the right hand side in \eqref{energy1}. The sum with respect to $j \le J$ 
makes the ${\mathcal Q}$ terms a telescopic sum, and we are left with the energy balance:
\begin{multline}
\label{energy2}
\sum_{j \le J} \, \Delta x \, (u_j^{n+1})^2 -\sum_{j \le J} \, \Delta x \, (u_j^n)^2 
=\sum_{j \le J} \, \sum_{\ell=1}^{p+r} \, d_\ell \, \Delta x \, (u_{j+\ell-r}^n -u_{j-r}^n)^2 \\
+\Delta x \, {\mathcal Q} 
(u_{J+2-r}^n-u_{J+1-r}^n,\dots,u_J^n-u_{J-1}^n,\underbrace{u_J^n}_{\text{\rm $r$-th entry}},u_{J+1}^n-u_J^n,\dots, 
u_{J+p}^n-u_{J+p-1}^n) \, .
\end{multline}
Let us start with the ``boundary term'' on the right hand side of \eqref{energy2}. Thanks to the property of 
${\mathcal Q}$ given in Lemma \ref{ippdiscrete}, we know that the coefficient of $(u_J^n)^2$ in the expression
$$
{\mathcal Q} (u_{J+2-r}^n-u_{J+1-r}^n,\dots,u_J^n-u_{J-1}^n,u_J^n,u_{J+1}^n-u_J^n,\dots,u_{J+p}^n-u_{J+p-1}^n) 
$$
is $-\lambda \, a<0$. Hence, by repeatedly applying the Young inequality for the cross terms in the quadratic 
form ${\mathcal Q}$, we get, for a suitable constant $C>0$ (recall the relation $\lambda \, \Delta x =\Delta t$):
\begin{multline}
\label{energy3}
\sum_{j \le J} \, \Delta x \, (u_j^{n+1})^2 -\sum_{j \le J} \, \Delta x \, (u_j^n)^2 +\dfrac{a}{2} \, \Delta t \, (u_J^n)^2 \\
\le \sum_{\ell=1}^{p+r} \, d_\ell \, \sum_{j \le J} \, \Delta x \, (u_{j+\ell-r}^n -u_{j-r}^n)^2 
+C \, \Delta t \, \sum_{\ell=2-r}^p (u_{J+\ell}^n-u_{J+\ell-1}^n)^2 \, .
\end{multline}
It remains to estimate the ``bulk'' term on the right hand side of \eqref{energy3}, which encodes the 
``$\ell^2$-dissipation'' of the discretization \eqref{transportnum} of the transport equation. More precisely, if 
we start from a sequence $v \in \ell^2(\Z)$ and consider its image by the contraction \eqref{contraction}, we 
can rewrite Equation \eqref{energy1} for $v$ and use the decomposition given in 
Lemma \ref{ippdiscrete} to derive the inequality:
\begin{equation}
\label{energy4}
\forall \, v \in \ell^2(\Z) \, ,\quad 
\sum_{j \in \Z} \Delta x \left( \sum_{\ell=-r}^p a_\ell \, v_{j+\ell} \right)^2 - \sum_{j \in \Z} \Delta x \left( v_j \right)^2 
=\sum_{\ell=1}^{p+r} \, d_\ell \, \sum_{j \in \Z} \, \Delta x \, (v_{j+\ell-r} -v_{j-r})^2 \le 0 \, .
\end{equation}
Let us now consider the sequence $(u_j^n)_{j \le J+p} \in \ell^2$ for the given integer $n \in \N$. We extend 
the latter sequence as an $\ell^2$ sequence on the whole set of integers $\Z$ by symmetry with respect to 
the index $J+p$:
$$
\forall \, j \in \Z \, ,\quad v_j := \begin{cases}
u_j^n \, ,& \text{\rm if } j \le J+p \, ,\\
u_{2\, (J+p)-j}^n \, ,& \text{\rm if } j \ge J+p \, .
\end{cases}
$$
We then compute
\begin{multline}
\label{energy5}
\underbrace{\sum_{\ell=1}^{p+r} \, d_\ell \, \sum_{j \in \Z} \, \Delta x \, (v_{j+\ell-r} -v_{j-r})^2}_{\le \, 0 \, 
\text{\rm by \eqref{energy4}}} \\
=2 \sum_{\ell=1}^{p+r} \, d_\ell \, \sum_{j \le J} \, \Delta x \, (u_{j+\ell-r}^n -u_{j-r}^n)^2 
+\sum_{\ell=1}^{p+r} \, d_\ell \, \Delta x \, \sum_{j=J+1}^{J+2\, (p+r)-\ell-1} \, (v_{j+\ell-r} -v_{j-r})^2.
\end{multline}
For $\ell=1,\dots,p+r$ and $j=J+1,\dots,J+2\, (p+r)-\ell-1$, $v_{j+\ell-r}$ and $v_{j-r}$ on the right hand side 
of \eqref{energy5} are taken among the real numbers $u_{J+1-r}^n,\dots,u_{J+p}^n$. We therefore obtain 
the upper bound:
$$
\left| \sum_{\ell=1}^{p+r} \, d_\ell \, \Delta x \, \sum_{j=J+1}^{J+2\, (p+r)-\ell-1} \, (v_{j+\ell-r} -v_{j-r})^2 \right| 
\le C \, \Delta t \, \sum_{\ell=2-r}^p \, (u_{J+\ell}^n-u_{J+\ell-1}^n)^2 \, .
$$
Using the latter bound in \eqref{energy5}, we get the estimate
$$
\sum_{\ell=1}^{p+r} \, d_\ell \, \sum_{j \le J} \, \Delta x \, (u_{j+\ell-r}^n -u_{j-r}^n)^2 
\le C \, \Delta t \, \sum_{\ell=2-r}^p \, (u_{J+\ell}^n-u_{J+\ell-1}^n)^2 \, ,
$$
which simplifies \eqref{energy3} into:
$$
\sum_{j \le J} \, \Delta x \, (u_j^{n+1})^2 -\sum_{j \le J} \, \Delta x \, (u_j^n)^2 
+\dfrac{a}{2} \, \Delta t \, (u_J^n)^2 \le C \, \Delta t \, \sum_{\ell=2-r}^p \, (u_{J+\ell}^n-u_{J+\ell-1}^n)^2 \, .
$$

We now apply the summation argument of Lemma \ref{gronwall} to the latter inequality and derive the estimate:
\begin{multline*}
\sup_{n \in \N} \, \left( {\rm e}^{-2\, \gamma \, n\, \Delta t} \, \sum_{j \le J} \, \Delta x \, (u_j^n)^2 \right) 
+\sum_{n\ge 0} \, \Delta t \, {\rm e}^{-2\, \gamma \, n\, \Delta t} \, (u_J^n)^2 \\
\le C \, \left\{ \sum_{j \le J} \, \Delta x \, (f_j)^2 
+\sum_{n\ge 0} \, \Delta t \, {\rm e}^{-2\, \gamma \, n\, \Delta t} \, 
\sum_{\ell=2-r}^p \, (u_{J+\ell}^n-u_{J+\ell-1}^n)^2 \right\} \, .
\end{multline*}
We then combine the latter inequality with the preliminary estimate \eqref{estimthm2-3}, which yields:
\begin{multline}
\label{estimthm2-4}
\sup_{n \in \N} \, \left( {\rm e}^{-2\, \gamma \, n\, \Delta t} \, \sum_{j \le J} \, \Delta x \, (u_j^n)^2 \right) 
+\sum_{n\ge 0} \, \Delta t \, {\rm e}^{-2\, \gamma \, n\, \Delta t} \, (u_J^n)^2 \\
\le C \left\{ \sum_{j \le J} \, \Delta x \, (f_j)^2 +\sum_{n\ge 0} \, \Delta t \, 
{\rm e}^{-2\, \gamma \, n\, \Delta t} \, \sum_{\ell=1}^p \, (g_{J+\ell}^n)^2 \right\} \, .
\end{multline}

At this stage, we have almost proved that \eqref{estimation} holds up to the index $k_b+1$. Indeed, if 
we combine the trace estimates provided by \eqref{estimthm2-3} and \eqref{estimthm2-4}, we get a full 
control of the trace of $(u^n)$, that is of $(u_{J+\ell}^n)_{n \ge 0}$ at $\ell=-r-k_b,\dots,p$:
\begin{equation}
\label{estimthm2-5}
\sum_{n\ge 0} \, \Delta t \, {\rm e}^{-2\, \gamma \, n\, \Delta t} \, \sum_{\ell=-r-k_b}^p (u_{J+\ell}^n)^2 
\le C \left\{ \sum_{j \le J+p} \, \Delta x \, (f_j)^2 +\sum_{n\ge 1} \, \Delta t \, 
{\rm e}^{-2\, \gamma \, n\, \Delta t} \, \sum_{\ell=1}^p \, (g_{J+\ell}^n)^2 \right\} \, .
\end{equation}
Combining with \eqref{estimthm2-4}, we have completed the proof of \eqref{estimation} for the index $k_b+1$, 
which also completes the proof of Theorem \ref{thm2}.
\end{proof}

\subsection{Convergence estimates for Neumann outflow numerical boundary conditions}

In the previous paragraph, we have proved the stability estimate \eqref{estimation} in order to highlight 
the fact that our method automatically yields the verification of the Uniform Kreiss Lopatinskii Condition. 
However, the exponential weights arising in \eqref{estimation} and the fact that no ``interior'' source term 
is considered in \eqref{pbauxk} make this estimate hardly applicable as such in view of the convergence 
analysis below. We therefore state a slightly weakened but more practical version of Theorem \ref{thm2} 
which will help us proving Theorem \ref{thm3} below.

\begin{proposition}
\label{prop1}
Let $a>0$, let $k\ge 1$ and $k_b \in \N$, let $\lambda > 0$. Then there exists a constant $C>0$ such that for 
all initial data $(f_j)_{j \le J} \in \ell^2$, for any $N \in \N^*$, for all boundary source terms $(g_{J+1}^n)_{0 \le n \le N-1}, 
\dots, (g_{J+p}^n)_{0 \le n \le N-1}$, and all 
interior source terms $(F_j^n)_{j \le J,1\le n \le N}$ with $(F_j^n)_{j \le J}  \in \ell^2$ for all $n=1,\dots,N$,
the solution $(u_j^n)_{j \le J,0 \le n \le N}$ to the scheme:
\begin{equation}
\label{pbauxk'}
\begin{cases}
u_j^0=f_j \, ,& j\le J \, , \\
(D_-^{k_b} u^{n})_{J+\ell} =g_{J+\ell}^{n} \, ,& n=0,\dots,N-1 \, ,\quad \ell=1,\dots,p \, ,\\
u_j^{n+1} ={\displaystyle \sum_{\ell=-r}^p} a_\ell \, u_{j+\ell}^n +\Delta t \, F_j^{n+1} \, ,& 
n=0,\dots,N-1 \, ,\quad j \le J \, ,
\end{cases}
\end{equation}
where the coefficients $a_\ell$ satisfy Assumption \ref{as:scheme} with integer $k$, satisfies the estimate:
\begin{equation}
\label{estimation'}
\sup_{n=0,\dots,N} \, \sum_{j \le J} \, \Delta x \, (u_j^n)^2 \le C \left\{ \sum_{j \le J} \, \Delta x \, (f_j)^2 
+(N \, \Delta t)^2 \, \sup_{n=1,\dots,N} \,  \sum_{j \le J} \, \Delta x \, (F_j^n)^2 
+\sum_{n=0}^{N-1} \, \Delta t \, \sum_{\ell=1}^p (g_{J+\ell}^n)^2 \right\} \, 
\end{equation}
for any $\Delta t$ and $\Delta x$ such that $\Delta t/\Delta x = \lambda$ and $\Delta t \in (0, 1]$.
\end{proposition}

\begin{proof}
By linearity of \eqref{pbauxk'}, it is sufficient to examine separately the cases $F \equiv 0$ (no interior source term), 
and $f \equiv 0$, $g \equiv 0$ (interior forcing only). In the case $F \equiv 0$, the estimate \eqref{estimation'} is directly 
obtained by:
\begin{itemize}
   \item first extending the boundary source terms $(g_{J+\ell}^n)_{\ell=1,\dots,p}$ by $0$ for $n>N-1$, which does not affect 
   the solution to \eqref{pbauxk'} for $j \le J$ at time steps earlier than $N$,
   \item then passing to the limit $\gamma \to 0$ in \eqref{estimation} (and forgetting about the nonnegative trace estimate on 
   the left hand side of \eqref{estimation}).
\end{itemize}

\noindent It therefore remains to examine the case $f \equiv 0$, $g \equiv 0$, which is done by using the Duhamel formula. 
Namely, Theorem \ref{thm2} shows that the solution $(v_j^n)_{j \le J+p, n \in \N}$ to the numerical scheme with homogeneous 
boundary conditions:
\begin{equation*}
\begin{cases}
v_j^0=f_j \, ,& j\le J \, , \\
(D_-^{k_b} v^{n})_{J+\ell} =0 \, ,& n \in \N \, ,\quad \ell=1,\dots,p \, ,\\
v_j^{n+1} ={\displaystyle \sum_{\ell=-r}^p} a_\ell \, v_{j+\ell}^n \, ,& n \in \N \, ,\quad j \le J \, ,
\end{cases}
\end{equation*}
satisfies the uniform in time bound 
$$
\sup_{n \in \N} \, \sum_{j \le J} \, \Delta x \, (v_j^n)^2 \le C \, \sum_{j \le J} \, \Delta x \, (f_j)^2 \, .
$$
Writing $v^n$ under the form ${\mathcal S}^n \, f$, this means that the operator ${\mathcal S}$, which is the generator of the 
discrete semigroup $({\mathcal S}^n)_{n \in \N}$, is power bounded on $\ell^2(-\infty,J)$. At this stage, it remains to observe 
that the solution to \eqref{pbauxk'} in the case $f \equiv 0$, $g \equiv 0$, can be written with the Duhamel formula:
$$
\forall \, n =0,\dots,N \, ,\quad u^n =\sum_{m=1}^n \Delta t \, {\mathcal S}^{n-m} \, F^m \, .
$$
Since ${\mathcal S}$ is power bounded on $\ell^2$, we end up with:
$$
\sup_{n=0,\dots,N} \, \left( \sum_{j \le J} \, \Delta x \, (u_j^n)^2 \right)^{1/2} 
\le C \, \sum_{n=1}^N \Delta t \, \left( \sum_{j \le J} \, \Delta x \, (F_j^n)^2 \right)^{1/2} 
\le C \, N \, \Delta t \, \sup_{n=1,\dots,N} \,  \left( \sum_{j \le J} \, \Delta x \, (F_j^n)^2 \right)^{1/2} \, .
$$
This completes the proof of \eqref{estimation'}.
\end{proof}

\noindent We are now ready to prove our convergence result for the Neumann boundary condition at an outflow boundary.

\begin{theorem}
\label{thm3}
Let $a>0$, let $k \in \N^*$ and $k_b \in \N$, let $\lambda > 0$. Then there exists $C>0$ such that for any $T > 0$ and 
$J \in \N^*$, for any $L \geq 1$, and for any $u_0 \in H^{k+1}((-\infty,L))$, the solution $(u_j^n)_{j\le J,n \in \N}$ to
\begin{equation}
\label{schemademidroite}
\begin{cases}
u_j^0=u_j^{\rm in} :=\dfrac{1}{\Delta x} \, \int_{x_{j-1}}^{x_j} u_0(y) \, {\rm d}y \, ,& j\le J \, , \\
(D_-^{k_b} u^{n})_{J+\ell} =0 \, ,& 0 \le n \leq T / \Delta t - 1 \, ,\quad \ell=1,\dots,p \, ,\\
u_j^{n+1} ={\displaystyle \sum_{\ell=-r}^p} a_\ell \, u_{j+\ell}^n \, ,& 0 \le n \leq T/\Delta t -1 \, ,\quad j \le J \, ,
\end{cases}
\end{equation}
where the coefficients $a_\ell$ satisfy Assumption \ref{as:scheme} with integer $k$, satisfies, with $k_0 := \min(k,k_b)$:
\begin{equation}
\label{estiml2}
\sup_{0 \le n \le T/\Delta t} \, \left( \sum_{j \le J} \, \Delta x \, \left( 
u_j^n -\dfrac{1}{\Delta x} \, \int_{x_{j-1}}^{x_j} u_0(y-a\, t^n) \, {\rm d}y \right)^2 \right)^{1/2} 
\le C \, \big( \sqrt{T}+T \big) \, \Delta x^{k_0} \, \| u_0 \|_{H^{k_0+1}((-\infty,L))} \, ,
\end{equation}
where $\Delta x = L/J$ and $\Delta t = \lambda \, \Delta x$ (assumed to be less than $1$). 
In particular, the following global $\ell^\infty$ convergence estimate is satisfied:
\begin{equation}
\label{estimlinfty1}
\sup_{0 \le n \le T/\Delta t} \, \sup_{j \le J} \, \left| u_j^n -\dfrac{1}{\Delta x} \, 
\int_{x_{j-1}}^{x_j} u_0(y-a\, t^n) \, {\rm d}y 
\right| \le C \, \big( \sqrt{T}+T \big) \, \Delta x^{k_0 -1/2} \, \| u_0 \|_{H^{k_0+1}((-\infty,L))} \, .
\end{equation}
\end{theorem}

\begin{proof}
To pass from \eqref{estiml2} to \eqref{estimlinfty1} is very crude, and simply relies on the inequality:
$$
\forall \, j \le J \, ,\quad |b_j| \le \Delta x^{-1/2} \, \left( \sum_{j \le J} \, \Delta x \, b_j^2 \right)^{1/2} 
$$
for any real sequence $(b_j)_{j \leq J}$. We therefore focus on the derivation of the $\ell^\infty_n (\ell^2_j)$ estimate 
\eqref{estiml2} which unsurprisingly relies on the stability estimate \eqref{estimation'} provided by Proposition \ref{prop1}. 
In the proof below, we assume $k_b \le k$, so that the limiting order of convergence arises from the numerical boundary 
conditions in \eqref{schemademidroite} and not from the discretization of the transport equation. The proof in the case 
$k_b>k$ is quite similar and we leave the corresponding modifications to the interested reader. Since the validity of 
Assumption \ref{as:scheme} for some integer $k \ge 1$ implies the validity of the relations \eqref{eq:consist} for the 
restricted subset of indices $m=0,\dots,k_b$, we can even assume without loss of generality $k_b=k$.

We now denote by $U_0$ the extension of $u_0$ as a function in $H^{k+1}(\R)$ by the linear continuous reflexion 
operator of \cite{dautraylions}, and define, for any $j \in \Z$,
\begin{equation}
\label{defwnj}
w_j^n := \dfrac{1}{\Delta x} \, \int_{x_{j-1}}^{x_j} U_0(y-a\, t^n) \, {\rm d}y \, ,
\end{equation}
to be the cell average of the exact solution $((t,x) \longmapsto U_0(x-a\, t))$ to the transport equation over $\R$.

For the rest of this proof, we will denote the integer part of $T/\Delta t$ by $N_T$. Then for all $n = 0, \dots, N_T$ and 
$j \le J+p$, we define the error $\eps_j^n :=u_j^n-w_j^n$. By definition of the sequences $(u_j^n)$ and $(w_j^n)$, the 
error $(\eps_j^n)_{j \le J,n=0,\dots,N_T}$ is a solution to:
\begin{equation}
\label{erreurdemidroite}
\begin{cases}
\eps_j^0=\eps_j^{\rm in} \, ,& j\le J \, , \\
(D_-^{k_b} \eps^{n})_{J+\ell} =-(D_-^{k_b} w^{n})_{J+\ell} \, ,& 0 \le n \le N_T-1 \, ,\quad \ell=1,\dots,p \, ,\\
\eps_j^{n+1} ={\displaystyle \sum_{\ell=-r}^p} a_\ell \, \eps_{j+\ell}^n +\Delta t \, e_j^{n+1} \, ,& 
0 \le n \le N_T -1 \, ,\quad j \le J \, ,
\end{cases}
\end{equation}
where the initial condition $(\eps_j^{\rm in})_{j\le J}$ in \eqref{erreurdemidroite} satisfies $\eps_j^{\rm in} = 0$ for any $j \le J$, 
and the interior consistency error $(e_j^n)_{j \le J,n=1,\dots,N_T}$ in \eqref{erreurdemidroite} is given by
\begin{equation}
\label{consistencyerror}
e_j^n :=-\dfrac{1}{\Delta t} \, \left( w_j^n -\sum_{\ell=-r}^p a_\ell \, w_{j+\ell}^{n-1} \right) 
\end{equation}
for any $1 \le n \le N_T$ and any $j \le J$. To prove the convergence estimate \eqref{estimlinfty1}, the first task is to evaluate 
the size of the consistency errors in \eqref{erreurdemidroite}. The boundary errors in \eqref{erreurdemidroite} are dealt with by 
the following elementary result.

\begin{lemma}
\label{lem1}
Let $m,p \in \N$. There exists a constant $C > 0$ such that for any $\Delta x > 0$ and for any $v \in H^{m+1}((-\infty,L+p \, \Delta x))$, 
defining 
$$
v_j :=\dfrac{1}{\Delta x} \, \int_{x_{j-1}}^{x_j} v(y) \, {\rm d}y \, , \quad j \le J+p \, , 
$$
one has 
$$
\big| (D_-^m v)_{J+\ell} \big| \le C \, \Delta x^m \, \| v^{(m)} \|_{H^1((-\infty,L+p \, \Delta x))} \, ,\quad \ell =1,\dots,p \, .
$$
\end{lemma}

\begin{proof}[Proof of Lemma \ref{lem1}]
For $\ell = 1,\dots,p$, there holds
\begin{align*}
(D_-^m v)_{J+\ell} &= \sum_{m'=0}^m \binom{m}{m'} \, (-1)^{m-m'} \, \dfrac{1}{\Delta x} \, 
\int_{x_{J-m-1+\ell}}^{x_{J-m+\ell}} v(y+m'\, \Delta x) \, {\rm d}y \\
&= \sum_{m'=0}^m \binom{m}{m'} \, \dfrac{(-1)^{m-m'}}{(m-1)!} \, \dfrac{1}{\Delta x} \, 
\int_{x_{J-m-1+\ell}}^{x_{J-m+\ell}} \int_y^{y+m'\, \Delta x} v^{(m)}(z) \, (y+m'\, \Delta x-z)^{m-1} \, {\rm d}z \, {\rm d}y \, ,
\end{align*}
where we have used Taylor's formula and cancellation properties of the binomial coefficients. Applying the Cauchy-Schwarz 
inequality to each double integral in the latter expression, we get
\begin{align*}
|(D_-^m v)_{J+\ell}| &\le C \, \sum_{m'=0}^m  \left( \int_{x_{J-m-1+\ell}}^{x_{J-m+\ell}} \int_y^{y+m'\, \Delta x} v^{(m)}(z)^2 \, 
(y+m'\, \Delta x-z)^{2\, m-2} \, {\rm d}z \, {\rm d}y \right)^{1/2} \\
&\le C \, \Delta x^{m-1/2} \, \left( \int_{x_{J-m-1+\ell}}^{x_{J+\ell}} v^{(m)}(z)^2 \, {\rm d}z \right)^{1/2} 
\le C \, \Delta x^m \, \| v^{(m)} \|_{L^\infty (-\infty,L + p\Delta x)} \, .
\end{align*}
The proof follows by using the imbedding of $H^1$ in $L^\infty$ in one space dimension.
\end{proof}

We now apply Lemma \ref{lem1} with $m = k_b$ to evaluate the boundary errors in \eqref{erreurdemidroite}. We get:
\begin{align}
\sum_{n=0}^{N_T-1} \Delta t \, \sum_{\ell=1}^p \big( (D_-^{k_b} \varepsilon^n)_{J+\ell} \big)^2 \, & = 
\sum_{n=0}^{N_T-1} \Delta t \, \sum_{\ell=1}^p \big( (D_-^{k_b} w^n)_{J+\ell} \big)^2 \notag \\
&\le \, C \, N_T \, \Delta t \, \Delta x^{2\, k_b} \, \| U_0(\cdot - a t^n) \|_{H^{k_b+1}((-\infty,L + p\Delta x))}^2  \notag \\
& \le C \, T \, \Delta x^{2\, k_b} \, \| u_0 \|_{H^{k_b+1}((-\infty,L))}^2 \label{estimerreur2}
\end{align}
thanks to the continuity of the reflexion operator. Using \eqref{estimerreur2} in the stability estimate \eqref{estimation'} provided by 
Proposition \ref{prop1}, we already get
\begin{equation}
\label{estimerreur3}
\sup_{n=0,\dots,N_T} \, \sum_{j \le J} \, \Delta x \, (\eps_j^n)^2 
\le C \left\{ T^2 \, \sup_{n=1,\dots,N_T} \,  \sum_{j \le J} \, \Delta x \, (e_j^n)^2  + \Delta x^{2\, k_b} \, T \, 
\| u_0 \|_{H^{k_b+1}((-\infty,L))}^2 \right\} \, ,
\end{equation}
so the remaining task is to evaluate the size of the interior consistency error $(e_j^n)_{j \le J}$.

By its definition \eqref{defwnj}, the value of $w_j^n$ in each interior cell is the average of the extension of the exact solution, 
that is the cell average of $U_0(\cdot - a \, t^n)$. We thus have, for any $j \leq J$ and $n \ge 1$:
\begin{align*}
e_j^n &= \dfrac{1}{\Delta t} \, \left(  \dfrac{1}{\Delta x} \, \int_{x_{j-1}}^{x_j} U_0(y -a\, \Delta t-a\, t^{n-1}) \, {\rm d}y 
-\sum_{\ell=-r}^p \dfrac{1}{\Delta x} \, \int_{x_{j-1}}^{x_j} a_\ell \, U_0(y+\ell \, \Delta x - a\, t^{n-1}) \, {\rm d}y \right) \\
&= \dfrac{1}{\Delta t \, \Delta x} \, \int_{x_{j-1}}^{x_j} \left( U_0 (y-a\, \Delta t -a\, t^{n-1}) -\sum_{\ell=-r}^p 
a_\ell \, U_0 (y+\ell \, \Delta x -a\, t^{n-1}) \right) {\rm d}y \, ,\\
&=  \dfrac{1}{\Delta t \, \Delta x} \, \int_{x_{j-1}}^{x_j} \left( U^{n-1} (y-a\, \Delta t) -\sum_{\ell=-r}^p 
a_\ell \, U^{n-1} (y+\ell \, \Delta x ) \right) {\rm d}y 
\end{align*}
where we defined $U^{n-1} \in H^{k+1}(\R)$ as $U_0(\cdot - a\, t^{n-1})$. We also define the consistency error on the whole domain: 
$$
\forall \, j \in \Z \, ,\quad E_j^n :=\dfrac{1}{\Delta t \, \Delta x} \, \int_{x_{j-1}}^{x_j} \left( 
U^{n-1} (y-a\, \Delta t) -\sum_{\ell=-r}^p a_\ell \, U^{n-1} (y+\ell \, \Delta x) \right) {\rm d}y \, ,
$$
so that we have $e_j^n=E_j^n$ for all $j \le J$. We now use Fourier analysis and derive the estimate: 
\begin{align*}
\sum_{j \le J} \Delta x \, (e_j^n)^2 \le \sum_{j \in \Z} \Delta x \, (E_j^n)^2 
&\le \dfrac{1}{\Delta t^2} \, \sum_{j \in \Z} \dfrac{1}{\Delta x} \, \left( \int_{x_{j-1}}^{x_j} 
\left(U^{n-1} (y-a\, \Delta t) -\sum_{\ell=-r}^p a_\ell \, U^{n-1} (y+\ell \, \Delta x)\right) \, {\rm d}y \right)^2 \\
&\le \dfrac{1}{\Delta t^2} \, \int_\R \left( U^{n-1} (y-a\, \Delta t) -\sum_{\ell=-r}^p a_\ell \, 
U^{n-1} (y+\ell \, \Delta x) \right)^2 {\rm d}y \\
&\le \dfrac{1}{2\, \pi \, \Delta t^2} \, \int_\R \left| \widehat{U^{n-1}}(\xi) \right|^2 \, \left| 
{\rm e}^{-i\, \lambda \, a\, \Delta x \, \xi} -\sum_{\ell=-r}^p a_\ell \, {\rm e}^{i\, \ell\, \Delta x \, \xi} \right|^2 \, {\rm d}\xi \, . 
\end{align*}
We now use Assumption \ref{as:scheme} (recall $k_b=k$), to derive the bound:
$$
\left| {\rm e}^{-i\, \lambda \, a\, \Delta x \, \xi} -\sum_{\ell=-r}^p a_\ell \, {\rm e}^{i\, \ell\, \Delta x \, \xi} \right| 
\le C\, (\Delta x \, |\xi|)^{k_b+1} \, ,
$$
uniformly with respect to $\Delta x$ and $\xi \in \R$. Recalling that the ratio $\Delta t/\Delta x$ is kept fixed, 
we end up with the consistency estimate:
\begin{equation}
\label{estimerreur4}
\sum_{j \le J} \Delta x \, (e_j^n)^2 \le C \, \Delta x^{2\, k_b} \, \| U^{n-1} \|_{H^{k_b+1}(\R)}^2 
= C \, \Delta x^{2\, k_b} \, \| U_0 \|_{H^{k_b+1}(\R)}^2 
\le C \, \Delta x^{2\, k_b} \, \| u_0 \|_{H^{k_b+1}((-\infty,L))}^2 \, ,
\end{equation}
where we have used the fact that $U^{n-1}$ has been constructed as an extension of $u_0(\cdot-a\, t^{n-1})$ 
boundedly with respect to the $H^{k_b+1}$ norm. Going back to \eqref{estimerreur3} and using the estimate 
\eqref{estimerreur4}, we finally derive
\begin{equation*}
\sup_{n=0,\dots,N_T} \, \sum_{j \le J} \, \Delta x \, (\eps_j^n)^2 
\le C \, \Delta x^{2\, k_b} \, (T+T^2) \, \| u_0 \|_{H^{k_b+1}((-\infty,L))}^2 \, ,
\end{equation*}
which completes the proof of \eqref{estiml2}.
\end{proof}

\subsection{Convergence estimates for Dirichlet numerical boundary conditions}

This short paragraph is devoted to the complementary boundary value problem on a half-line that can be extracted from 
\eqref{transportnum}, \eqref{dirichlet}, \eqref{neumann}, and that focuses on the (homogeneous) Dirichlet boundary condition 
\eqref{dirichlet} at the inflow boundary $x=0$. Hence the half-line is now $\R^+$. To be consistent with the notations introduced 
in Section \ref{sect:2}, we let $x_j :=j \, \Delta x$, $j \in \Z$. The ghost cells correspond to the intervals $(x_{-r},x_{1-r}), \dots, 
(x_{-1},x_0)$, see Figure \ref{fig:maillage''}. The time step $\Delta t$ is linked to the space step by keeping $\Delta t/\Delta x 
=\lambda$, $\lambda$ being the same fixed positive constant as in Section \ref{sect:2}. We prove the following result, that is 
the analogue of Theorem \ref{thm3}.

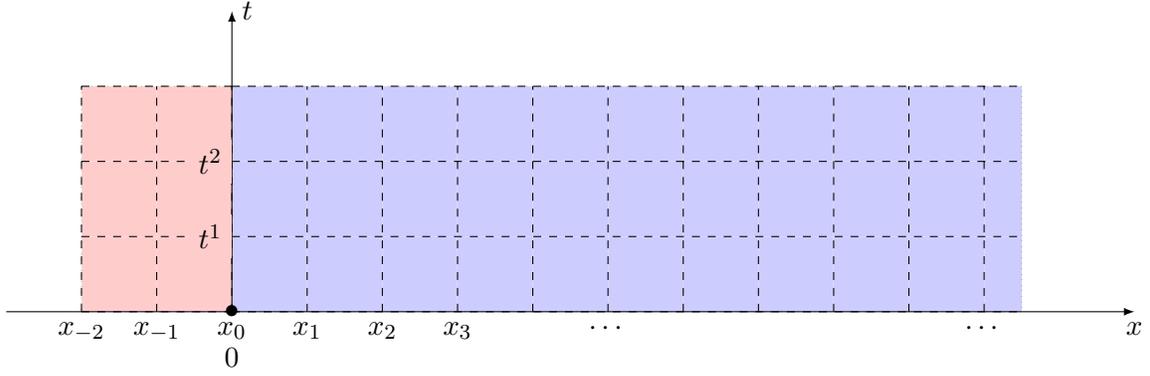
\begin{figure}[h!]
\begin{center}
\begin{tikzpicture}[scale=2,>=latex]
\draw [ultra thin, dotted, fill=blue!20] (0,0) rectangle (5.25,1.5);
\draw [thin, dashed, fill=white] (0,0) grid [step=0.5] (5.25,1.5);
\draw [ultra thin, dotted, fill=red!20] (-1,0) rectangle (0,1.5);
\draw [thin, dashed, fill=white] (-1,0) grid [step=0.5] (0,1.5);
\draw[black,->] (-1.5,0) -- (6,0) node[below] {$x$};
\draw[black,->] (0,0)--(0,2) node[right] {$t$};
\draw (0,0.5) node[left, fill=red!20]{$t^1$};
\draw (0,1) node[left, fill=red!20]{$t^2$};
\draw (-1,0) node[below]{$x_{-2}$};
\draw (-0.5,0) node[below]{$x_{-1}$};
\draw (0,0) node[below]{$x_0$};
\draw (0,-0.3) node {$0$};
\draw (0.5,0) node[below]{$x_1$};
\draw (1,0) node[below]{$x_2$};
\draw (1.5,0) node[below]{$x_3$};
\draw (2.5,0) node[below]{$\cdots$};
\draw (5,0) node[below]{$\cdots$};
\node (centre) at (0,0){$\bullet$};
\end{tikzpicture}
\caption{The mesh on $\R^+$ in blue, and the ``ghost cells'' in red ($r=2$ here).}
\label{fig:maillage''}
\end{center}
\end{figure}

\begin{proposition}
\label{prop2}
Let $a>0$, let $k \in \N^*$, let $\lambda > 0$. There exists $C>0$ such that for any $T>0$ and $J \in \N^*$, for any $L \geq 1$ 
and for any  $u_0 \in H^{k+1}((0,+\infty))$ satisfying the flatness conditions:
$$
u_0(0)=\cdots =u_0^{(k)}(0) =0 \, ,
$$
the solution $(u_j^n)_{j\ge 1,n \in \N}$ to the scheme: 
\begin{equation}
\label{schemademidroite'}
\begin{cases}
u_j^0=u_j^{\rm in} :=\dfrac{1}{\Delta x} \, \int_{x_{j-1}}^{x_j} u_0(y) \, {\rm d}y \, ,& j \ge 1, \\
u^{n}_\ell =0 \, ,& 0\leq n \leq T/\Delta t -1 \, ,\quad \ell=1-r,\dots,0 \, ,\\
u_j^{n+1} ={\displaystyle \sum_{\ell=-r}^p} a_\ell \, u_{j+\ell}^n \, ,& 0\leq n \leq T/\Delta t -1 \, ,\quad j \ge 1 \, ,\\
\end{cases}
\end{equation}
where the coefficients $a_\ell$ satisfy Assumption \ref{as:scheme} with integer $k$, satisfies 
\begin{equation}
\label{estiml2'}
\sup_{0 \le n \le T/\Delta t} \, \left( \sum_{j \ge 1} \, \Delta x \, \left( 
u_j^n -\dfrac{1}{\Delta x} \, \int_{x_{j-1}}^{x_j} u_0(y-a\, t^n) \, {\rm d}y \right)^2 \right)^{1/2} 
\le C \, T \, \Delta x^k \, \| u_0 \|_{H^{k+1}((0,+\infty))}  
\end{equation}
where $\Delta x = L/J$ and $\Delta t = \lambda \, \Delta x$. It is understood in \eqref{estiml2'} that the initial condition $u_0$ 
has been extended by $0$ to $\R^-$. In particular, the following global $\ell^\infty$ convergence estimate is satisfied:
\begin{equation}
\label{estimlinfty1'}
\sup_{0 \le n \le T/\Delta t} \, \sup_{j \ge 1} \, \left| u_j^n -\dfrac{1}{\Delta x} \, \int_{x_{j-1}}^{x_j} u_0(y-a\, t^n) \, {\rm d}y 
\right| \le C \, T \, \Delta x^{k-1/2} \, \| u_0 \|_{H^{k+1}((0,+\infty))} \, .
\end{equation}
\end{proposition}

The proof follows more or less the exact same arguments as the proof of Theorem \ref{thm3}. One first defines the 
cell average of the exact solution $((t,x) \mapsto u_0(x-a\, t))$ to the transport equation in the half space $\R^+$ with 
homogeneous Dirichlet boundary condition at $x=0$. The flatness assumption on $u_0$ ensures that extending $u_0$ 
by $0$ on $\R^-$ yields a function in $H^{k+1}(\R)$. The error between the solution to \eqref{schemademidroite'} and 
the cell average of the exact solution will satisfy a system of the form \eqref{schemademidroite'} with \emph{zero} initial 
condition and \emph{homogeneous} Dirichlet boundary condition but with a nonzero interior forcing term (the so-called 
consistency error). Estimating this error is done by using a similar (and even easier) extension procedure as in the proof 
of Theorem \ref{thm3} and Fourier analysis. We leave the details to the interested reader. The stability estimate for 
\eqref{schemademidroite'} is analogous to that of Proposition \ref{prop1} and follows from Theorem \ref{thm2} (in the 
case $k_b=0$) since we have already observed that dealing with this case does not necessitate the consistency conditions 
\eqref{eq:consist} but only the stability assumption \eqref{eq:stab}.

\section{Proof of Theorem \ref{thm1}}
\label{sect:4}

\subsection{Stability estimates on a finite interval}\label{stable_operator}

We now turn to the study of the numerical scheme \eqref{transportnum}, \eqref{dirichlet}, \eqref{neumann}, 
which is an iteration in a finite dimensional space and therefore really corresponds to a numerical scheme 
that can be implemented in practice. Let us recall that the space step $\Delta x$ is given by $\Delta x =L/J$ 
with $J$ a positive integer, the spatial grid is defined by means of the points $x_j := j\, \Delta x$, $j \in \Z$, and 
the time step is given by $\Delta t =\lambda \, \Delta x$ with $\lambda$ a fixed positive constant. We first prove 
a stability estimate for \eqref{transportnum}, \eqref{dirichlet}, \eqref{neumann}, which will have various consequences.

\begin{proposition}
\label{prop3}
Let $a>0$, let $k\ge 1$ and $k_b \in \N$, let $\lambda > 0$. Then there exists a constant $C_0>0$ such that for all initial 
data $(f_1,\dots,f_J)$, the solution $(u_j^n)_{1 \le j \le J,n \in \N}$ to the scheme:
\begin{equation}
\label{schemainterval}
\begin{cases}
u_j^0=f_j \, ,& j=1,\dots,J, \\
u^{n}_\ell =0 \, ,& n \in \N \, ,\quad \ell=1-r,\dots,0 \, ,\\
(D_-^{k_b} u^{n})_{J+\ell} =0 \, ,& n \in \N \, ,\quad \ell=1,\dots,p \, ,\\
u_j^{n+1} ={\displaystyle \sum_{\ell=-r}^p} a_\ell \, u_{j+\ell}^n \, ,& n \in \N \, ,\quad j=1,\dots,J \, ,\\
\end{cases}
\end{equation}
where the coefficients $a_\ell$ satisfy Assumption \ref{as:scheme} with integer $k$, 
satisfies
\begin{equation}
\label{estiminterval}
\forall \, n \in \N \, ,\quad \left( \sum_{j=1}^J \Delta x \, (u_j^n)^2 \right)^{1/2} 
\le C_0 \, {\rm e}^{C_0 \, n \, \Delta t/L} \, \left( \sum_{j=1}^{J} \Delta x \, (f_j)^2 \right)^{1/2} \, .
\end{equation}
\end{proposition}

The estimate \eqref{estiminterval} is compatible with the limit $L \rightarrow +\infty$ (and $\Delta x$, $\Delta t$ 
fixed) where we formally recover either the stability estimate \eqref{estimation'} or its analogue for Dirichlet 
boundary conditions on a half-line, at least in the case of homogeneous boundary conditions and no interior source 
term. More important, the estimate \eqref{estiminterval} also has a consequence in terms of the spectral radius 
of the perturbed Toeplitz matrix associated with \eqref{schemainterval}. In the absence of the outflow boundary 
at $x=L$, the numerical scheme \eqref{schemademidroite'} with homogeneous Dirichlet boundary conditions at 
the inflow boundary $x=0$ corresponds to iterating the Toeplitz operator represented by the semi-infinite matrix:
$$
\begin{pmatrix}
a_0 & \cdots & a_p & 0 & \cdots & \cdots & \cdots \\
\vdots & \ddots & & \ddots & \ddots & & \\
a_{-r} & \cdots & a_0 & \cdots & a_p & \ddots & \\
0 & \ddots & & \ddots & & \ddots & \ddots \\
\vdots & \ddots & \ddots & & \ddots & & \ddots \\
\end{pmatrix} \, .
$$
Acting on $\ell^2(\N)$, the spectrum of this operator is known to contain at least the closed curve $\{ \sum 
a_\ell \exp (i\, \ell \, \theta),\\ \theta \in [0,2\, \pi] \}$, see \cite[Chapter 7]{TE} and references therein for 
a precise statement. In particular, Assumption \ref{as:scheme} and \cite[Theorem 7.1]{TE} show that the 
spectrum of this operator lies inside the closed unit disk $\{ z \in \C \, , \, |z| \le 1 \}$ and accumulates at 
the point $1$. Truncating the above semi-infinite matrix to make it of size $J$ corresponds to prescribing 
homogeneous Dirichlet boundary conditions at the outflow boundary. In that case, the resulting operator 
is of the form $\Pi \, T \, \Pi$ with $\Pi$ an orthogonal projection and $T$ a contraction. In particular, 
prescribing homogeneous Dirichlet boundary conditions at the outflow boundary leaves the spectrum 
(consisting now of finitely many eigenvalues) within the closed unit disk.

Prescribing Neumann numerical boundary conditions as in \eqref{schemainterval} corresponds to first truncating 
the semi-infinite matrix to make it of size $J$, and then making ``large'' perturbations of the coefficients in the 
$k_b$ last columns of the last $p$ rows of the truncated matrix. For instance, if $r=p=1$ and $k_b=1$, resp. 
$k_b=2$, the corresponding matrix reads:
$$
\begin{pmatrix}
a_0 & a_1 & 0 & \cdots & 0 \\
a_{-1} & a_0 & a_1 & \ddots & \vdots \\
0 & \ddots & \ddots & \ddots & 0 \\
\vdots & \ddots & \ddots & a_0 & a_1 \\
0 & \cdots & 0 & a_{-1} & a_0 +a_1 \end{pmatrix} \, , \quad \text{\rm resp.} \quad \begin{pmatrix}
a_0 & a_1 & 0 & \cdots & 0 \\
a_{-1} & a_0 & a_1 & \ddots & \vdots \\
0 & \ddots & \ddots & \ddots & 0 \\
\vdots & \ddots & a_{-1} & a_0 & a_1 \\
0 & \cdots & 0 & a_{-1}-a_1 & a_0 +2\, a_1 \end{pmatrix} \, .
$$
It is not obvious at first sight that the spectrum of the matrix ${\mathcal A}_J \in {\mathcal M}_J(\R)$ associated 
with the iteration \eqref{schemainterval} will not deviate from the closed unit disk at a distance $O(1)$, or even 
$O(J^{-\alpha})$ for some $\alpha \in (0,1)$, giving rise to instabilities for \eqref{schemainterval}. However, the 
estimate \eqref{estiminterval} has an important consequence for the matrix ${\mathcal A}_J$ since it shows in 
particular that its $\ell^2$ induced norm verifies
$$
\forall \, n \in \N \, ,\quad \| {\mathcal A}_J^n \|_2 \le C \, {\rm e}^{C \, n \, \Delta t/L} \, ,
$$
so taking the $1/n$-th power and passing to the limit (recalling $\Delta t/L=\lambda/J$, $\lambda>0$ constant), 
the spectral radius of ${\mathcal A}_J$ satisfies
$$
\rho ({\mathcal A}_J) \le \exp (C/J) \, .
$$
Hence the spectrum of ${\mathcal A}_J$ does not deviate from the closed unit disk at a distance larger than 
$O(1/J)$. (The bound in \eqref{estiminterval} is even more precise since it even predicts the behavior of the 
$\varepsilon$-pseudospectrum of ${\mathcal A}_J$, see \cite{TE}, prohibiting in particular Jordan blocks associated 
with eigenvalues of largest modulus.)

\begin{proof}[Proof of Proposition \ref{prop3}]
The derivation of \eqref{estiminterval} follows from the finite speed of propagation of the numerical scheme 
\eqref{transportnum}. (Observe that our argument below does not extend to implicit discretizations of the transport 
equation.) More precisely, let us assume for simplicity that $J$ is even. Let
$$
N_0 :=\min \left( {\rm E} \left( \dfrac{J}{2\, p} \right) \, , \, {\rm E} \left( \dfrac{J/2 -k_b}{r} \right) \right) \, .
$$
Then for $n \le N_0$, we can write the solution to \eqref{schemainterval} as the superposition of solutions 
to two initial boundary value problems of the form \eqref{schemademidroite} and \eqref{schemademidroite'}. 
Indeed, if the initial condition $(f_j = u_j^{\rm in})_{j \le J}$ in \eqref{schemademidroite} vanishes for $j \le J/2$, 
then the solution to \eqref{schemademidroite} satisfies the homogeneous Dirichlet boundary condition:
$$
\forall \, n \le N_0 \, ,\quad \forall \, \ell=1-r,\dots,0 \, ,\quad u_\ell^n =0 \, ,
$$
because the support of $u^{\rm in}$ is shifted of $p$ cells to the left at each time iteration. In particular, the 
restriction of the solution to the half-line problem \eqref{schemademidroite} with some initial condition $u^{\rm in}$ 
vanishing for $j \le J/2$ satisfies the numerical scheme \eqref{schemainterval} up to $n=N_0$. Similarly, if the 
initial condition $(u_j^{\rm in})_{j \ge 1}$ in \eqref{schemademidroite'} vanishes for $j \ge J/2+1$, then the solution 
to the half line problem \eqref{schemademidroite'} satisfies:
$$
\forall \, n \le N_0 \, ,\quad \forall \, \ell=1,\dots,p \, ,\quad (D_-^{k_b} \, u^n)_{J+\ell} =0 \, ,
$$
because the support of $u^{\rm in}$ is shifted of $r$ cells to the right at each time iteration, and the 
definition of $N_0$ ensures
$$
\dfrac{J}{2} +1+N_0 \, r \le J+1-k_b \, ,
$$
so the solution to \eqref{schemademidroite'} vanishes at all points that are used in the calculation of the 
finite differences $(D_-^{k_b} \, u^n)_{J+\ell}$. Truncating the initial condition for \eqref{schemainterval} 
to the left and to the right of the medium cell $J/2$, we can, thanks to the linearity of the scheme, 
combine the stability estimate \eqref{estimation'} and the analogous one for Dirichlet 
boundary conditions to show that the solution to \eqref{schemainterval} satisfies:
$$
\sup_{n=0,\dots,N_0} \, \sum_{j=1}^J \Delta x \, (u_j^n)^2 \le C \, \sum_{j=1}^{J} \Delta x \, (f_j)^2 
\, ,
$$
with a constant $C$ that is independent of $L$ and $J$. It then remains to iterate the latter estimate on 
each interval of integers $\{ 0,\dots,N_0 \}$, $\{ N_0,\dots,2\, N_0 \}$ and so on to derive \eqref{estiminterval}.
\end{proof}

\subsection{Convergence}

We now prove Theorem \ref{thm1}, whose proof combines the convergence estimates given by Theorem 
\ref{thm3} and Proposition \ref{prop2}, and the stability estimate of Proposition \ref{prop3}. Let us recall that 
we consider for the continuous problem {transport} an initial condition $u_0 \in H^{k+1}((0,L))$ that satisfies 
the flatness conditions:
$$
u_0(0) =\cdots =u_0^{(k)}(0)=0 \, .
$$
We therefore extend $u_0$ by zero on $\R^-$, considering now $u_0$ as an element of $H^{k+1}((-\infty,L))$ 
that vanishes on $\R^-$. Let us consider a ${\mathcal C}^\infty$ function $\chi$ on $\R$ that satisfies:
$$
\chi(x) =\begin{cases}
0 \, ,& \text{\rm if $x \le 1/3$,} \\
1 \, ,& \text{\rm if $x \ge 2/3$,}
\end{cases}
$$
and then decompose $u_0$ as follows:
$$
\forall \, x <L \, ,\quad u_0(x) =\underbrace{\chi (x/L) \, u_0 (x)}_{=: \, v_0(x)} 
+\underbrace{(1-\chi (x/L)) \, u_0 (x)}_{=: \, w_0(x)} \, .
$$
The function $v_0$ belongs to $H^{k+1}((-\infty,L))$ and vanishes for $x \le L/3$. Conversely, the function 
$w_0$ belongs to $H^{k+1}((-\infty,L))$ and vanishes for $x \ge 2\, L/3$. Moreover, using $L \ge 1$, there 
holds uniformly in $L$:
$$
\| v_0 \|_{H^{k+1}((-\infty,L))} \le C \, \| u_0 \|_{H^{k+1}((0,L))} \, ,\quad 
\| w_0 \|_{H^{k+1}((-\infty,L))} \le C \, \| u_0 \|_{H^{k+1}((0,L))} \, ,
$$
where $C$ depends only on the fixed function $\chi$ (and on $k$).

Let $N \in \N$ denote the largest integer such that
$$
N \, p \le \dfrac{J}{3} \, ,\quad \mbox{and} \quad N\, r \le \dfrac{J}{3} -k_b \, .
$$
We assume that $J$ is large enough so that $N \ge 1$. Applying the same finite speed of propagation 
argument as in the proof of Proposition \ref{prop3}, we can write
$$
\forall \, n=0,\dots,N \, ,\quad \forall \, j=1,\dots,J \, ,\quad u_j^n=v_j^n+w_j^n \, ,
$$
where $(v_j^n)_{j \le J,n=0,\dots,N}$ satisfies an iteration of the form \eqref{schemademidroite}, namely:
\begin{equation}
\label{schemavjn}
\begin{cases}
v_j^0=v_j^{\rm in} \, ,& j\le J \, ,\\
(D_-^{k_b} v^{n})_{J+\ell} =0 \, ,& n=0,\dots,N-1 \, ,\quad \ell=1,\dots,p \, ,\\
v_j^{n+1} ={\displaystyle \sum_{\ell=-r}^p} a_\ell \, v_{j+\ell}^n \, ,& n=0,\dots,N-1 \, ,\quad j \le J \, ,\\
\end{cases}
\end{equation}
and $(w_j^n)_{j \ge 1,n=0,\dots,N}$ satisfies an iteration of the form \eqref{schemademidroite'}, namely:
\begin{equation}
\label{schemawjn}
\begin{cases}
w_j^0=w_j^{\rm in} \, ,& j \ge 1 \, \\
w^{n}_\ell =0 \, ,& n=0,\dots,N-1 \, ,\quad \ell=1-r,\dots,0 \, ,\\
w_j^{n+1} ={\displaystyle \sum_{\ell=-r}^p} a_\ell \, w_{j+\ell}^n \, ,& n=0,\dots,N-1 \, ,\quad j \ge 1 \, .\\
\end{cases}
\end{equation}
The initial conditions in \eqref{schemavjn} and \eqref{schemawjn} are defined as follows. Following 
\eqref{initial1}, we define the discretized initial condition associated with $v_0$ for \eqref{schemavjn} by 
letting\footnote{Recall here that $u_0$ has been extended (by zero) to $\R^-$ so $v_0$ is well defined 
on $(-\infty,L)$.}:
\begin{equation*}
\forall \, j \le J \, ,\quad v_j^{\rm in} :=\dfrac{1}{\Delta x} \, \int_{x_{j-1}}^{x_j} v_0(y) \, {\rm d}y \, .
\end{equation*}
To define the initial condition for \eqref{schemawjn}, let us recall that $1-\chi(\cdot/L)$ vanishes on the 
interval $(2\, L/3,L)$ so we can extend $w_0$ by zero to $(L,+\infty)$ and view $w_0$ as an element 
of $H^{k+1}((0, \infty))$ that satisfies the flatness conditions of Proposition \ref{prop2}. We then define 
the initial condition for \eqref{schemawjn} by letting:
\begin{equation*}
\forall \, j \ge 1 \, ,\quad w_j^{\rm in} :=\dfrac{1}{\Delta x} \, \int_{x_{j-1}}^{x_j} w_0(y) \, {\rm d}y \, .
\end{equation*}

We can apply Theorem \ref{thm3} to \eqref{schemavjn} and get the convergence estimate:
\begin{multline}
\sup_{0 \le n \le N} \, \left( \sum_{j=1}^J \, \Delta x \, \left( 
v_j^n -\dfrac{1}{\Delta x} \, \int_{x_{j-1}}^{x_j} v_0(y-a\, t^n) \, {\rm d}y \right)^2 \right)^{1/2} \\
\le C \, (\sqrt{N\, \Delta t} + N\, \Delta t) \, \Delta x^{k_0} \, \| v_0 \|_{H^{k+1}((-\infty,L))} \\
\le C \, (\sqrt{N\, \Delta t} + N\, \Delta t) \, \Delta x^{k_0} \, \| u_0 \|_{H^{k+1}((0,L))} \, ,\label{estimthm1-1}
\end{multline}
with $k_0:=\min (k,k_b)$ and where the estimate is uniform with respect to $L$. Applying Proposition 
\ref{prop2} to \eqref{schemawjn}, we get the other convergence estimate:
$$
\sup_{0 \le n \le N} \, \left( \sum_{j=1}^J \, \Delta x \, \left( 
w_j^n -\dfrac{1}{\Delta x} \, \int_{x_{j-1}}^{x_j} w_0(y-a\, t^n) \, {\rm d}y \right)^2 \right)^{1/2} 
\le C \, N \, \Delta t \, \Delta x^k \, \| u_0 \|_{H^{k+1}((0,L))} \, ,
$$
which we combine with \eqref{estimthm1-1} to derive:
\begin{multline}
\label{estimthm1-2}
\sup_{0 \le n \le N} \, \left( \sum_{j=1}^J \, \Delta x \, \left( 
u_j^n -\dfrac{1}{\Delta x} \, \int_{x_{j-1}}^{x_j} u_0(y-a\, t^n) \, {\rm d}y \right)^2 \right)^{1/2}  \\
\le C_1 \, (\sqrt{N\, \Delta t} + N\, \Delta t) \, \Delta x^{k_0} \, \| u_0 \|_{H^{k+1}((0,L))} \, .
\end{multline}
Here we use a specific notation $C_1$ for the constant in \eqref{estimthm1-2} in order to emphasize 
its role in the concluding argument below (in the same way, we use a specific notation $C_0$ for the 
constant in the stability estimate \eqref{estiminterval}).

It remains more or less to iterate in time \eqref{estimthm1-2}. Namely, at the time level $N$, \eqref{estimthm1-2} 
shows that we can write:
$$
\forall \, j=1,\dots,J \, ,\quad u_j^N =\dfrac{1}{\Delta x} \, \int_{x_{j-1}}^{x_j} u_0(y-a\, t^N) \, {\rm d}y +\eps_j \, ,
$$
with
$$
\left( \sum_{j=1}^J \, \Delta x \, \eps_j^2 \right)^{1/2} 
\le C_1 \, (\sqrt{N\, \Delta t} + N\, \Delta t)  \, \Delta x^{k_0} \, \| u_0 \|_{H^{k+1}((0,L))} \, .
$$
To iterate the process, we define $(u_j^{N,n})_{1 \le j \le J, 0 \le n N}$ by 
\[
\begin{cases}
\displaystyle{} u_j^{N,0} = \frac{1}{\Delta x} \int_{x_{j-1}}^{x_j} u_0(y - a t^N)\,{\rm d}y \, \, ,& j=1,\dots,J, \\
u^{N,n}_\ell =0 \, ,& 0 \le n \le N-1 \, ,\quad \ell=1-r,\dots,0 \, ,\\
(D_-^{k_b} u^{N,n})_{J+\ell} =0 \, ,& 0 \le n \le N-1 \, ,\quad \ell=1,\dots,p \, ,\\
u_j^{N,n+1} ={\displaystyle \sum_{\ell=-r}^p} a_\ell \, u_{j+\ell}^{N,n} \, ,& 0 \le n \le N-1 \, ,\quad j=1,\dots,J \, .\\
\end{cases}
\]
Note that $u_0(\cdot - a \, t^N)$ satisfies the flatness conditions of Theorem \ref{thm1}, and that:
$$
\| u_0(\cdot - a t^N) \|_{H^{k+1}((0,L))} \le \| u_0 \|_{H^{k+1}((0,L))} \, .
$$
Thus, on the one side, \eqref{estimthm1-2} applies and we have 
\begin{multline*}
\sup_{0 \le n \le N} \, \left( \sum_{j=1}^J \, \Delta x \, \left( 
u_j^{N,n} -\dfrac{1}{\Delta x} \, \int_{x_{j-1}}^{x_j} u_0(y-a(t^N + t^n) \, {\rm d}y \right)^2 \right)^{1/2} \\
\le C_1 \, (\sqrt{N\, \Delta t} + N\, \Delta t) \, \Delta x^{k_0} \, \| u_0 \|_{H^{k+1}((0,L))} \, .
\end{multline*}
On the other side, the sequence $(\delta_j^n := u_j^{N+n} - u_j^{N,n})_{1 \le j \le J, 0 \le n N}$ satisfies: 
\[
\begin{cases}
\delta_j^{0} = \eps_j \, ,& j=1,\dots,J, \\
\delta^{n}_\ell =0 \, ,& 0 \le n \le N-1 \, ,\quad \ell=1-r,\dots,0 \, ,\\
(D_-^{k_b} \delta^{n})_{J+\ell} =0 \, ,& 0 \le n \le N-1 \, ,\quad \ell=1,\dots,p \, ,\\
\delta_j^{n+1} ={\displaystyle \sum_{\ell=-r}^p} a_\ell \, \delta_{j+\ell}^{n} \, ,& 0 \le n \le N-1 \, ,\quad j=1,\dots,J \, \\
\end{cases}
\]
so that, thanks to the stability estimate \eqref{estiminterval} of Proposition \ref{prop3}, for any $n \leq N$, 
\begin{multline*}
\left( \sum_{j=1}^J \Delta x \, (\delta_j^n)^2 \right)^{1/2} 
\le C_0 \, {\rm e}^{C_0 \, n \, \Delta t/L} \, \left( \sum_{j=1}^{J} \Delta x \, (\eps_j)^2 \right)^{1/2} \\
\leq C_0 \, {\rm e}^{C_0 \, n \, \Delta t/L} \, C_1 \, (\sqrt{N\, \Delta t} + N\, \Delta t) \, \Delta x^{k_0} \, \| u_0 \|_{H^{k+1}((0,L))} \, .
\end{multline*}

Finally, thanks to the triangle inequality, 
\begin{multline*}
\sup_{0 \le n \le N} \, \left( \sum_{j=1}^J \, \Delta x \, \left( 
u_j^{N+n} -\dfrac{1}{\Delta x} \, \int_{x_{j-1}}^{x_j} u_0(y-a(t^N + t^n) \, {\rm d}y \right)^2 \right)^{1/2} \\
\le C_1 (1 + C_0 \, {\rm e}^{C_0 \, N \, \Delta t/L}) \, (\sqrt{N\, \Delta t} + N\, \Delta t) \, 
\Delta x^{k_0} \, \| u_0 \|_{H^{k+1}((0,L))} \, .
\end{multline*}

Applying iteratively the same argument, we end up with
\begin{align*}
\forall \, m \in \N \, ,\quad \sup_{0 \le n \le m\, N} \, & \left( \sum_{j=1}^J \, \Delta x \, \left( 
u_j^n -\dfrac{1}{\Delta x} \, \int_{x_{j-1}}^{x_j} u_0(y-a\, t^n) \, {\rm d}y \right)^2 \right)^{1/2} \\
&\le C_1 \,\sum_{\mu=0}^{m-1} \left( C_0 \, {\rm e}^{C_0 \, N \, \Delta t/L} \right)^\mu 
(\sqrt{N\, \Delta t} + N\, \Delta t) \, \Delta x^{k_0} \, \| u_0 \|_{H^{k+1}((0,L))} \\
&\le C_1 \, C_0^m \, {\rm e}^{C_0 \, m\, N \, \Delta t/L} \, (\sqrt{N\, \Delta t} + N\, \Delta t) \,  \Delta x^{k_0} \, 
\| u_0 \|_{H^{k+1}((0,L))} \, ,
\end{align*}
where we have assumed $C_0 \ge 2$ without loss of generality. It remains to choose $m:={\rm E} (N\Delta t/T)+1$, 
which by definition of $N$ is uniformly bounded with respect to $J$ ($N$ scales like $c \, J$ with $c>0$ constant, and 
$\Delta t$ scales like $c'/J$), and we end up with:
$$
\sup_{0 \le n \le T/\Delta t} \, \left( \sum_{j=1}^J \, \Delta x \, \left( 
u_j^n -\dfrac{1}{\Delta x} \, \int_{x_{j-1}}^{x_j} u_0(y-a\, t^n) \, {\rm d}y \right)^2 \right)^{1/2} 
\le C \, (\sqrt{T}+T) \, {\rm e}^{C\, T/L} \, \Delta x^{k_0} \, \| u_0 \|_{H^{k+1}((0,L))} \, .
$$
The convergence estimate \eqref{estimlinfty} of Theorem \ref{thm1} follows by a direct lower bound for the 
norm on the left hand side.

\section{Numerical experiments}

In this section, we illustrate with numerical examples the results proved in the paper. 
We consider the second order in time and space Lax-Wendroff scheme, that writes 
\begin{equation}
\label{laxwendroff}
\frac{u_j^{n+1} - u_j^n}{\Delta t} +a \, \frac{u_{j+1}^n - u_{j-1}^n}{2 \, \Delta x} = 
\frac{a^2 \, \Delta t}{2} \, \frac{u_{j+1}^n - 2 \, u_j^n + u_{j-1}^n}{\Delta x^2}, \quad n \in \N \, , \quad j = 1,\dots,J \, ,
\end{equation}
which means, in the present formalism, that $r = p = 1$ and $a_{-1} =(a^2 \, \lambda^2 +a \, \lambda)/2$, 
$a_0 =1-a^2 \, \lambda^2$, $a_1 =(a^2 \, \lambda^2 -a \, \lambda)/2$ and $k = 2$. This scheme is stable 
(over $\Z$) and second order accurate provided that $\lambda \, |a| \leq 1$, and we choose 
$a=1$ and 
$\lambda = 0.7$ in the numerical simulations reported below.

We compare the results with the first order Neumann outflow condition ($k_b = 1$) and the second order one 
($k_b = 2$), for various initial conditions. Namely, we consider $L=1$ for the length of the space interval and 
we choose as initial data, 
\[
\begin{array}{l}
u^{0,1}(x) = ((x - 1/2)^+)^3, \quad x \in (0,1) \, , \\
u^{0,2}(x) = ((x - 1/2)^+)^{2.6}, \quad x \in (0,1) \, , \\
u^{0,3}(x) = ((x - 1/2)^+)^{2.5}, \quad x \in (0,1) \, ,
\end{array}
\]
which satisfy $u^{0,1} \in H^{3}$, $u^{0,2} \in H^{3}$ and $u^{0,3} \in H^{3-\epsilon}$ for any $\epsilon > 0$. The 
final time is $T = 0.5$ (after $T$ the exact solution vanishes). The convergence result provided in this paper applies 
for the initial data $u^{0,1}$ and $u^{0,2}$, and forecasts that the $l^\infty$ error is bounded by $C \, \Delta x^{1.5}$ 
if $k_b = 2$ and $C \, \Delta x^{0.5}$ if $k_b=1$.

Figure \ref{figini} represents the initial condition $u^{0,2}$ on a grid with $40$ cells on $(0, 1)$. On figure \ref{fig025} 
and \ref{fig05}, we plot the numerical solutions (at different times) obtained with $k_b = 0$ (homogeneous Dirichlet 
outflow condition), $k_b = 1$ (homogeneous Neumann outflow condition) and $k_b = 2$ (homogeneous ``second order" 
Neumann outflow condition). As expected, the Dirichlet condition shows a larger boundary layer, and, especially at time 
$0.2625$, the solution with $k_b = 2$ is much nearer to the exact solution than the others. 

Let us now analyze more precisely the error of the schemes. 
In the following, the computed ``error" components actually are not the 
$u_j^n - \int_{x_{j-1}}^{x_j} u_0(y - a t^n) \, {\rm d}y/\Delta x$: 
for simplicity, they are replaced with the local values $u_j^n - u_0((x_{j-1} + x_{j})/2 - at^n)$. If $u_0 \in H^2$, 
$| u_0((x_{j-1} + x_{j})/2 - at^n) - \int_{x_{j-1}}^{x_j} u_0(y - a t^n) \, {\rm d}y/\Delta x |$ can be bounded by a constant times 
$\Delta x^2$, thus the present formula will not disrupt our evaluation of the order of the scheme, which cannot be greater 
than 2. 
By observing Table \ref{u1}, we see that:
\begin{itemize}
 \item with $k_b = 2$ the order of the scheme indeed seems to be $2$,
 \item while with $k_b = 1$, it seems to be $1$.
\end{itemize}
In Table \ref{u2} with a less smooth (but still in $H^3$) initial datum,
\begin{itemize}
 \item with $k_b = 2$ the order of the scheme seems to belong to $(1.7,1.75)$,
 \item and with $k_b = 1$, it seems to be $1$.
\end{itemize}
At last, Table \ref{u3} suggests that:
\begin{itemize}
 \item with $k_b = 2$ the order of the scheme be in $(1.65,1.7)$,
 \item and with $k_b = 1$, the order still be $1$.
\end{itemize}

Clearly the convergence result we have obtained here seem to be non-optimal and there is hope, for numerical schemes 
such as \eqref{laxwendroff} that fit into the framework of \cite{bbjfc}, to fill this discrepancy by using numerical boundary 
layer expansions. However, one clearly observes that the convergence order does depend not only on the smoothness 
of the solution but also on the order of the homogeneous Neumann condition for the outflow boundary.

In order to illustrate the discussion about the stability of the scheme \eqref{schemainterval} and the properties of the matrix 
${\mathcal A}_J \in {\mathcal M}_J(\R)$ (Section \ref{stable_operator}), we report in Table \ref{norms} the spectral radius and 
the $l^2$ induced norm of this matrix (both are computed approximately thanks to the functions \texttt{eig} and \texttt{norm} 
of the library numpy.linalg in the Python language), with respect to the boundary condition and the number of cells. We observe 
that in all cases the spectral radius is smaller than 1, but that with the second order Neumann condition, the norm of the matrix 
is larger than $1$. This means that stability estimates cannot be obtained by showing that the $\ell^2$-norm does not increase, 
hence the need for some well-designed analytical tool (here we have used an induction argument with respect to $k_b$ 
combined with the finite speed of propagation).

\begin{figure}
\begin{center}
\setlength{\unitlength}{0.240900pt}
\ifx\plotpoint\undefined\newsavebox{\plotpoint}\fi
\begin{picture}(1500,900)(0,0)
\sbox{\plotpoint}{\rule[-0.200pt]{0.400pt}{0.400pt}}%
\put(130.0,82.0){\rule[-0.200pt]{315.338pt}{0.400pt}}
\put(130.0,82.0){\rule[-0.200pt]{4.818pt}{0.400pt}}
\put(110,82){\makebox(0,0)[r]{$0$}}
\put(1419.0,82.0){\rule[-0.200pt]{4.818pt}{0.400pt}}
\put(130.0,179.0){\rule[-0.200pt]{315.338pt}{0.400pt}}
\put(130.0,179.0){\rule[-0.200pt]{4.818pt}{0.400pt}}
\put(110,179){\makebox(0,0)[r]{$0.02$}}
\put(1419.0,179.0){\rule[-0.200pt]{4.818pt}{0.400pt}}
\put(130.0,276.0){\rule[-0.200pt]{315.338pt}{0.400pt}}
\put(130.0,276.0){\rule[-0.200pt]{4.818pt}{0.400pt}}
\put(110,276){\makebox(0,0)[r]{$0.04$}}
\put(1419.0,276.0){\rule[-0.200pt]{4.818pt}{0.400pt}}
\put(130.0,373.0){\rule[-0.200pt]{315.338pt}{0.400pt}}
\put(130.0,373.0){\rule[-0.200pt]{4.818pt}{0.400pt}}
\put(110,373){\makebox(0,0)[r]{$0.06$}}
\put(1419.0,373.0){\rule[-0.200pt]{4.818pt}{0.400pt}}
\put(130.0,471.0){\rule[-0.200pt]{315.338pt}{0.400pt}}
\put(130.0,471.0){\rule[-0.200pt]{4.818pt}{0.400pt}}
\put(110,471){\makebox(0,0)[r]{$0.08$}}
\put(1419.0,471.0){\rule[-0.200pt]{4.818pt}{0.400pt}}
\put(130.0,568.0){\rule[-0.200pt]{315.338pt}{0.400pt}}
\put(130.0,568.0){\rule[-0.200pt]{4.818pt}{0.400pt}}
\put(110,568){\makebox(0,0)[r]{$0.1$}}
\put(1419.0,568.0){\rule[-0.200pt]{4.818pt}{0.400pt}}
\put(130.0,665.0){\rule[-0.200pt]{315.338pt}{0.400pt}}
\put(130.0,665.0){\rule[-0.200pt]{4.818pt}{0.400pt}}
\put(110,665){\makebox(0,0)[r]{$0.12$}}
\put(1419.0,665.0){\rule[-0.200pt]{4.818pt}{0.400pt}}
\put(130.0,762.0){\rule[-0.200pt]{315.338pt}{0.400pt}}
\put(130.0,762.0){\rule[-0.200pt]{4.818pt}{0.400pt}}
\put(110,762){\makebox(0,0)[r]{$0.14$}}
\put(1419.0,762.0){\rule[-0.200pt]{4.818pt}{0.400pt}}
\put(130.0,859.0){\rule[-0.200pt]{315.338pt}{0.400pt}}
\put(130.0,859.0){\rule[-0.200pt]{4.818pt}{0.400pt}}
\put(110,859){\makebox(0,0)[r]{$0.16$}}
\put(1419.0,859.0){\rule[-0.200pt]{4.818pt}{0.400pt}}
\put(130.0,82.0){\rule[-0.200pt]{0.400pt}{187.179pt}}
\put(130.0,82.0){\rule[-0.200pt]{0.400pt}{4.818pt}}
\put(130,41){\makebox(0,0){$0$}}
\put(130.0,839.0){\rule[-0.200pt]{0.400pt}{4.818pt}}
\put(392.0,82.0){\rule[-0.200pt]{0.400pt}{153.935pt}}
\put(392.0,762.0){\rule[-0.200pt]{0.400pt}{23.367pt}}
\put(392.0,82.0){\rule[-0.200pt]{0.400pt}{4.818pt}}
\put(392,41){\makebox(0,0){$0.2$}}
\put(392.0,839.0){\rule[-0.200pt]{0.400pt}{4.818pt}}
\put(654.0,82.0){\rule[-0.200pt]{0.400pt}{153.935pt}}
\put(654.0,762.0){\rule[-0.200pt]{0.400pt}{23.367pt}}
\put(654.0,82.0){\rule[-0.200pt]{0.400pt}{4.818pt}}
\put(654,41){\makebox(0,0){$0.4$}}
\put(654.0,839.0){\rule[-0.200pt]{0.400pt}{4.818pt}}
\put(915.0,82.0){\rule[-0.200pt]{0.400pt}{187.179pt}}
\put(915.0,82.0){\rule[-0.200pt]{0.400pt}{4.818pt}}
\put(915,41){\makebox(0,0){$0.6$}}
\put(915.0,839.0){\rule[-0.200pt]{0.400pt}{4.818pt}}
\put(1177.0,82.0){\rule[-0.200pt]{0.400pt}{187.179pt}}
\put(1177.0,82.0){\rule[-0.200pt]{0.400pt}{4.818pt}}
\put(1177,41){\makebox(0,0){$0.8$}}
\put(1177.0,839.0){\rule[-0.200pt]{0.400pt}{4.818pt}}
\put(1439.0,82.0){\rule[-0.200pt]{0.400pt}{187.179pt}}
\put(1439.0,82.0){\rule[-0.200pt]{0.400pt}{4.818pt}}
\put(1439,41){\makebox(0,0){$1$}}
\put(1439.0,839.0){\rule[-0.200pt]{0.400pt}{4.818pt}}
\put(130.0,82.0){\rule[-0.200pt]{0.400pt}{187.179pt}}
\put(130.0,82.0){\rule[-0.200pt]{315.338pt}{0.400pt}}
\put(1439.0,82.0){\rule[-0.200pt]{0.400pt}{187.179pt}}
\put(130.0,859.0){\rule[-0.200pt]{315.338pt}{0.400pt}}
\put(645,741){\makebox(0,0)[r]{Initial condition $u^{0,2}$}}
\put(665.0,741.0){\rule[-0.200pt]{24.090pt}{0.400pt}}
\put(146,82){\usebox{\plotpoint}}
\put(801,81.67){\rule{7.950pt}{0.400pt}}
\multiput(801.00,81.17)(16.500,1.000){2}{\rule{3.975pt}{0.400pt}}
\multiput(834.00,83.61)(6.937,0.447){3}{\rule{4.367pt}{0.108pt}}
\multiput(834.00,82.17)(22.937,3.000){2}{\rule{2.183pt}{0.400pt}}
\multiput(866.00,86.59)(3.604,0.477){7}{\rule{2.740pt}{0.115pt}}
\multiput(866.00,85.17)(27.313,5.000){2}{\rule{1.370pt}{0.400pt}}
\multiput(899.00,91.59)(2.145,0.488){13}{\rule{1.750pt}{0.117pt}}
\multiput(899.00,90.17)(29.368,8.000){2}{\rule{0.875pt}{0.400pt}}
\multiput(932.00,99.58)(1.486,0.492){19}{\rule{1.264pt}{0.118pt}}
\multiput(932.00,98.17)(29.377,11.000){2}{\rule{0.632pt}{0.400pt}}
\multiput(964.00,110.58)(1.113,0.494){27}{\rule{0.980pt}{0.119pt}}
\multiput(964.00,109.17)(30.966,15.000){2}{\rule{0.490pt}{0.400pt}}
\multiput(997.00,125.58)(0.829,0.496){37}{\rule{0.760pt}{0.119pt}}
\multiput(997.00,124.17)(31.423,20.000){2}{\rule{0.380pt}{0.400pt}}
\multiput(1030.00,145.58)(0.689,0.496){45}{\rule{0.650pt}{0.120pt}}
\multiput(1030.00,144.17)(31.651,24.000){2}{\rule{0.325pt}{0.400pt}}
\multiput(1063.00,169.58)(0.551,0.497){55}{\rule{0.541pt}{0.120pt}}
\multiput(1063.00,168.17)(30.876,29.000){2}{\rule{0.271pt}{0.400pt}}
\multiput(1095.58,198.00)(0.497,0.514){63}{\rule{0.120pt}{0.512pt}}
\multiput(1094.17,198.00)(33.000,32.937){2}{\rule{0.400pt}{0.256pt}}
\multiput(1128.58,232.00)(0.497,0.606){63}{\rule{0.120pt}{0.585pt}}
\multiput(1127.17,232.00)(33.000,38.786){2}{\rule{0.400pt}{0.292pt}}
\multiput(1161.58,272.00)(0.497,0.698){63}{\rule{0.120pt}{0.658pt}}
\multiput(1160.17,272.00)(33.000,44.635){2}{\rule{0.400pt}{0.329pt}}
\multiput(1194.58,318.00)(0.497,0.815){61}{\rule{0.120pt}{0.750pt}}
\multiput(1193.17,318.00)(32.000,50.443){2}{\rule{0.400pt}{0.375pt}}
\multiput(1226.58,370.00)(0.497,0.897){63}{\rule{0.120pt}{0.815pt}}
\multiput(1225.17,370.00)(33.000,57.308){2}{\rule{0.400pt}{0.408pt}}
\multiput(1259.58,429.00)(0.497,1.004){63}{\rule{0.120pt}{0.900pt}}
\multiput(1258.17,429.00)(33.000,64.132){2}{\rule{0.400pt}{0.450pt}}
\multiput(1292.58,495.00)(0.497,1.147){61}{\rule{0.120pt}{1.013pt}}
\multiput(1291.17,495.00)(32.000,70.899){2}{\rule{0.400pt}{0.506pt}}
\multiput(1324.58,568.00)(0.497,1.219){63}{\rule{0.120pt}{1.070pt}}
\multiput(1323.17,568.00)(33.000,77.780){2}{\rule{0.400pt}{0.535pt}}
\multiput(1357.58,648.00)(0.497,1.341){63}{\rule{0.120pt}{1.167pt}}
\multiput(1356.17,648.00)(33.000,85.579){2}{\rule{0.400pt}{0.583pt}}
\multiput(1390.58,736.00)(0.497,1.464){63}{\rule{0.120pt}{1.264pt}}
\multiput(1389.17,736.00)(33.000,93.377){2}{\rule{0.400pt}{0.632pt}}
\put(146,82){\raisebox{-.8pt}{\makebox(0,0){$\Box$}}}
\put(179,82){\raisebox{-.8pt}{\makebox(0,0){$\Box$}}}
\put(212,82){\raisebox{-.8pt}{\makebox(0,0){$\Box$}}}
\put(245,82){\raisebox{-.8pt}{\makebox(0,0){$\Box$}}}
\put(277,82){\raisebox{-.8pt}{\makebox(0,0){$\Box$}}}
\put(310,82){\raisebox{-.8pt}{\makebox(0,0){$\Box$}}}
\put(343,82){\raisebox{-.8pt}{\makebox(0,0){$\Box$}}}
\put(375,82){\raisebox{-.8pt}{\makebox(0,0){$\Box$}}}
\put(408,82){\raisebox{-.8pt}{\makebox(0,0){$\Box$}}}
\put(441,82){\raisebox{-.8pt}{\makebox(0,0){$\Box$}}}
\put(474,82){\raisebox{-.8pt}{\makebox(0,0){$\Box$}}}
\put(506,82){\raisebox{-.8pt}{\makebox(0,0){$\Box$}}}
\put(539,82){\raisebox{-.8pt}{\makebox(0,0){$\Box$}}}
\put(572,82){\raisebox{-.8pt}{\makebox(0,0){$\Box$}}}
\put(605,82){\raisebox{-.8pt}{\makebox(0,0){$\Box$}}}
\put(637,82){\raisebox{-.8pt}{\makebox(0,0){$\Box$}}}
\put(670,82){\raisebox{-.8pt}{\makebox(0,0){$\Box$}}}
\put(703,82){\raisebox{-.8pt}{\makebox(0,0){$\Box$}}}
\put(735,82){\raisebox{-.8pt}{\makebox(0,0){$\Box$}}}
\put(768,82){\raisebox{-.8pt}{\makebox(0,0){$\Box$}}}
\put(801,82){\raisebox{-.8pt}{\makebox(0,0){$\Box$}}}
\put(834,83){\raisebox{-.8pt}{\makebox(0,0){$\Box$}}}
\put(866,86){\raisebox{-.8pt}{\makebox(0,0){$\Box$}}}
\put(899,91){\raisebox{-.8pt}{\makebox(0,0){$\Box$}}}
\put(932,99){\raisebox{-.8pt}{\makebox(0,0){$\Box$}}}
\put(964,110){\raisebox{-.8pt}{\makebox(0,0){$\Box$}}}
\put(997,125){\raisebox{-.8pt}{\makebox(0,0){$\Box$}}}
\put(1030,145){\raisebox{-.8pt}{\makebox(0,0){$\Box$}}}
\put(1063,169){\raisebox{-.8pt}{\makebox(0,0){$\Box$}}}
\put(1095,198){\raisebox{-.8pt}{\makebox(0,0){$\Box$}}}
\put(1128,232){\raisebox{-.8pt}{\makebox(0,0){$\Box$}}}
\put(1161,272){\raisebox{-.8pt}{\makebox(0,0){$\Box$}}}
\put(1194,318){\raisebox{-.8pt}{\makebox(0,0){$\Box$}}}
\put(1226,370){\raisebox{-.8pt}{\makebox(0,0){$\Box$}}}
\put(1259,429){\raisebox{-.8pt}{\makebox(0,0){$\Box$}}}
\put(1292,495){\raisebox{-.8pt}{\makebox(0,0){$\Box$}}}
\put(1324,568){\raisebox{-.8pt}{\makebox(0,0){$\Box$}}}
\put(1357,648){\raisebox{-.8pt}{\makebox(0,0){$\Box$}}}
\put(1390,736){\raisebox{-.8pt}{\makebox(0,0){$\Box$}}}
\put(1423,832){\raisebox{-.8pt}{\makebox(0,0){$\Box$}}}
\put(715,741){\raisebox{-.8pt}{\makebox(0,0){$\Box$}}}
\put(146.0,82.0){\rule[-0.200pt]{157.789pt}{0.400pt}}
\put(130.0,82.0){\rule[-0.200pt]{0.400pt}{187.179pt}}
\put(130.0,82.0){\rule[-0.200pt]{315.338pt}{0.400pt}}
\put(1439.0,82.0){\rule[-0.200pt]{0.400pt}{187.179pt}}
\put(130.0,859.0){\rule[-0.200pt]{315.338pt}{0.400pt}}
\end{picture}
\end{center}
\caption{Initial condition $u^{0,2}$ with $40$ cells in $(0, 1)$. }\label{figini}
\end{figure}
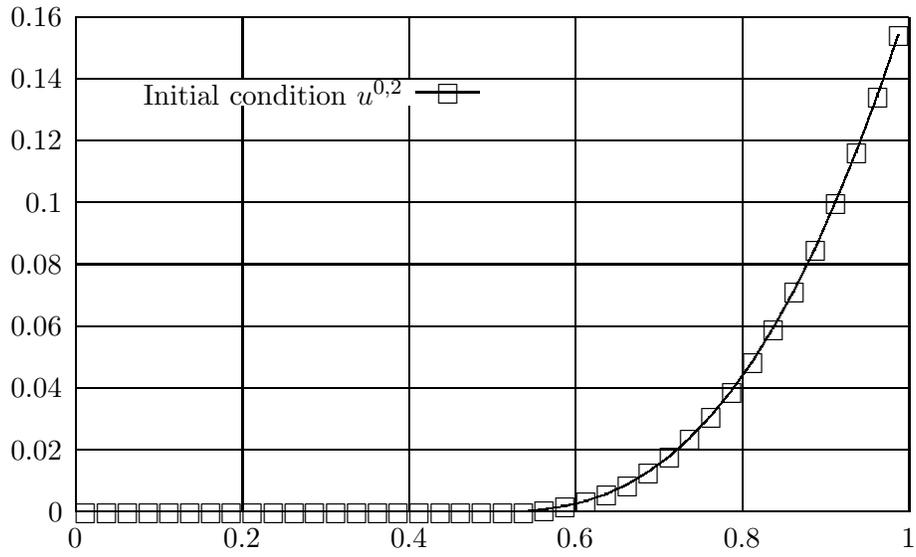

\begin{figure}
\begin{center}
\input{26_025}
\end{center}
\caption{Numerical and exact solutions at time $t = 0.2625$ with initial condition $u^{0,2}$ ($40$ cells). }
\label{fig025}
\end{figure}

\begin{figure}
\begin{center}
\input{26_05.tex}
\end{center}
\caption{Numerical and exact solutions at time $t = 0.5075$ with initial condition $u^{0,2}$ ($40$ cells). }
\label{fig05}
\end{figure}

\begin{table}[H]
\centering 
\begin{tabular}{|c|c|c|}
\hline
Number of cells $J$ & Measured error with $k_b = 2$ & Measured error with $k_b = 1$ \\
\hline
10 & 0.0025305 & 0.00833660625\\
\hline
20 & 0.0008281875 & 0.00491559140625\\
\hline
40 & 0.0002314921875 & 0.00262908841699\\
\hline
80 & 0.0000609287109375 & 0.0013994637865\\
\hline
160 & 0.0000156141357422 & 0.000720704311203\\
\hline
320 & 0.00000397348640443 & 0.000365563075521\\
\hline
640 & 0.00000100290833469 & 0.00018408024467\\
\hline
1280 & 0.000000251919175326 & 0.0000923642961781\\
\hline
\end{tabular} 
\caption{Error for datum $u^{0,1}$. }\label{u1}
\end{table}

\begin{table}[H]
\centering 
\begin{tabular}{|c|c|c|}
\hline
Number of cells $J$ & Measured error with $k_b = 2$ & Measured error with $k_b = 1$\\
\hline
10 & 0.00280385837572 & 0.0102978586289\\
\hline
20 & 0.000825428449649 & 0.00578352637669\\
\hline
40 & 0.000252680165957 & 0.00308529222599\\
\hline
80 & 0.0000781474537246 & 0.00161972927959\\
\hline
160 & 0.0000236164563317 & 0.000828965994226\\
\hline
320 & 0.00000711489098145 & 0.000419239010994\\
\hline
640 & 0.00000213591643874 & 0.000210806199835\\
\hline
1280 & 0.000000643052172999 & 0.000105699491246\\
\hline
\end{tabular} 
\caption{Error for datum $u^{0,2}$. }\label{u2}
\end{table}

\begin{table}[H]
\centering 
\begin{tabular}{|c|c|c|}
\hline
Number of cells $J$ & Measured error with $k_b = 2$ & Measured error with $k_b = 1$\\
\hline
10 & 0.00284239926561 & 0.0108072887024\\
\hline
20 & 0.00091837995271 & 0.00600083229228\\
\hline
40 & 0.000301806292425 & 0.00319976806911\\
\hline
80 & 0.0000975906472619 & 0.00167418222795\\
\hline
160 & 0.0000308167600202 & 0.000855523358729\\
\hline
320 & 0.00000972494981438 & 0.00043235384157\\
\hline
640 & 0.00000308448727156 & 0.000217323081725\\
\hline
1280 & 0.000000971185766911 & 0.000108947852591\\
\hline
\end{tabular} 
\caption{Error for datum $u^{0,3}$. }\label{u3}
\end{table}

\begin{table}[H]
\centering 
\begin{tabular}{|c|c|c|c|c|}
\hline
$J$ & Spectral radius, $k_b = 1$ & $l^2$ norm, $k_b = 1$ & Spectral radius, $k_b = 2$ & $l^2$ norm, $k_b = 2$\\
\hline
20 & 0.7100 & 0.9999 & 0.7098 & 1.0035 \\
\hline
80 & 0.74300 & 0.9999 & 0.7513 & 1.0035 \\
\hline
320 & 0.9208 & 0.9999 & 0.9212 & 1.0035 \\
\hline
1280 & 0.9817 & 0.9999 & 0.9805 & 1.0035 \\
\hline
\end{tabular} 
\caption{Spectral radii and $l^2$ induced norms of the linear operator associated with the scheme. }\label{norms}
\end{table}


\newpage
\appendix
\section{A discrete integration by parts lemma}
\label{appA}

In this appendix, we prove the following result, of which Lemma \ref{ippdiscrete} is an immediate corollary 
as explained below.

\begin{lemma}
\label{decompformequad}
Let $S \in {\mathcal M}_m(\R)$, $m \ge 2$, be a real symmetric matrix satisfying
$$
\sum_{i,j=1}^m S_{ij} =0 \, .
$$
Then there exists a unique real symmetric matrix $\widetilde{S}$ of size $m-1$, and some unique real 
numbers $d_1,\dots,d_{m-1}$, such that:
\begin{multline*}
S =\begin{pmatrix}
0 & 0 & \cdots & 0 \\
0 & & & & \\
\vdots & & \widetilde{S} & \\
0 & & & \end{pmatrix} -\begin{pmatrix}
 & & & 0 \\
 & \widetilde{S} & & \vdots \\
 & & & 0 \\
0 & \cdots & 0 & 0 \end{pmatrix} \\
+d_1 \, \begin{pmatrix}
1 & -1 & 0 & \cdots & 0 \\
-1 & 1 & \vdots & & \vdots \\
0 & \cdots & 0 & & \vdots \\
\vdots & & & & \vdots \\
0 & \cdots & \cdots & \cdots & 0 \end{pmatrix} +\cdots +d_{m-1} \, \begin{pmatrix}
1 & 0 & \cdots & 0 & -1 \\
0 & 0 & & 0 & 0 \\
\vdots & & & & \vdots \\
0 & 0 & & 0 & 0 \\
-1 & 0 & \cdots & 0 & 1 \end{pmatrix} \, .
\end{multline*}
\end{lemma}

\begin{proof}
The proof is completely elementary. The vector space of real symmetric matrices of size $m$ satisfying the 
condition
$$
\sum_{i,j=1}^m S_{ij} =0 \, ,
$$
has dimension $m\, (m+1)/2 -1=(m-1)\, (m+2)/2$. By standard linear algebra, it is therefore sufficient to prove 
that if a real symmetric matrix $\widetilde{S}$ of size $m-1$ and some real numbers $d_1,\dots,d_{m-1}$ 
satisfy:
\begin{multline}
\label{lemappA}
\begin{pmatrix}
0 & 0 & \cdots & 0 \\
0 & & & & \\
\vdots & & \widetilde{S} & \\
0 & & & \end{pmatrix} -\begin{pmatrix}
 & & & 0 \\
 & \widetilde{S} & & \vdots \\
 & & & 0 \\
0 & \cdots & 0 & 0 \end{pmatrix} \\
+d_1 \, \begin{pmatrix}
1 & -1 & 0 & \cdots & 0 \\
-1 & 1 & \vdots & & \vdots \\
0 & \cdots & 0 & & \vdots \\
\vdots & & & & \vdots \\
0 & \cdots & \cdots & \cdots & 0 \end{pmatrix} +\cdots +d_{m-1} \, \begin{pmatrix}
1 & 0 & \cdots & 0 & -1 \\
0 & 0 & & 0 & 0 \\
\vdots & & & & \vdots \\
0 & 0 & & 0 & 0 \\
-1 & 0 & \cdots & 0 & 1 \end{pmatrix} =0 \, ,
\end{multline}
then $\widetilde{S}=0$ and $d_1=\cdots=d_{m-1}=0$. Let us therefore assume that $\widetilde{S}$ and 
$d_1,\dots,d_{m-1}$ satisfy \eqref{lemappA}. Then considering the upper right coefficient, we first get 
$d_{m-1}=0$. Considering then the last column, we get $\widetilde{S}_{i,m-1}=0$ for all $i=1,\dots,m-1$. 
The proof follows by induction on $m$.
\end{proof}

\noindent Let us now explain how Lemma \ref{decompformequad} gives the result claimed in Lemma 
\ref{ippdiscrete}. On $\R^{p+r+1}$, with vectors written under the form $(v_{-r},\dots,v_p)$, we consider 
the quadratic form
$$
(v_{-r},\dots,v_p) \longmapsto 
2\, v_0 \, \Big( \sum_{\ell=-r}^p a_\ell \, v_\ell -v_0 \Big) +\Big( \sum_{\ell=-r}^p a_\ell \, v_\ell -v_0 \Big)^2 \, ,
$$
Because of \eqref{eq:consist} for $m=0$, the vector $(1,\dots,1)$ belongs to the isotropic cone of this 
quadratic form. Hence we can decompose the real symmetric matrix $S$ associated with this quadratic 
form by using Lemma \ref{decompformequad}. We get a decomposition of the form
\begin{equation*}
2\, v_0 \, \Big( \sum_{\ell=-r}^p a_\ell \, v_\ell -v_0 \Big) +\Big( \sum_{\ell=-r}^p a_\ell \, v_\ell -v_0 \Big)^2 
=\sum_{\ell=1}^{p+r} d_\ell \, (v_{\ell-r}-v_{-r})^2 
+\widetilde{{\mathcal Q}} \big( v_{1-r},\dots,v_p \big) -\widetilde{{\mathcal Q}} \big( v_{-r},\dots,v_{p-1} \big) \, ,
\end{equation*}
for some suitable real quadratic form $\widetilde{{\mathcal Q}}$ on $\R^{p+r}$. It only remains to make 
the invertible change of variables
$$
(v_{1-r},\dots,v_p) \longmapsto (v_{2-r}-v_{1-r},\dots,v_0-v_{-1},v_0,v_1-v_0,\dots,v_p-v_{p-1}) \, ,
$$
in the argument of $\widetilde{{\mathcal Q}}$, which modifies this quadratic form into some ${\mathcal Q}$, 
and we obtain the result of Lemma \ref{ippdiscrete} as announced. The value of ${\mathcal Q}$ on the 
$r$-th vector of the canonical basis of $\R^{p+r}$ is obtained by considering the specific sequence:
$$
\forall \, j \in \Z \, ,\quad v_j:=j \, ,
$$
and by identifying the dominant term in $j$.

\bibliographystyle{alpha}
\bibliography{CL}
\end{document}